\documentclass[12pt]{article}
\usepackage[utf8]{inputenc}
\usepackage[english]{babel}
\usepackage{algorithm}
\usepackage{algpseudocode}
\usepackage{nicefrac}
\usepackage{authblk}
\usepackage{srcltx}
\usepackage{color}
\usepackage{xcolor}
\usepackage{graphicx}
\usepackage{subcaption}
\usepackage{amsmath}
\usepackage{amssymb}
\usepackage{theorem}
\usepackage{euscript}
\usepackage{stmaryrd}
\usepackage{epic,eepic}
\usepackage{pstricks}
\usepackage{tikz} 
\usepackage{hyperref}
\usepackage{cleveref}
\usepackage{float}
\usepackage{multirow, array}
\usepackage{dsfont}
\usepackage{mathrsfs}
\usepackage{mathdots}
\usepackage{mathtools}
\usepackage{stackrel}
\usepackage{relsize}
\usepackage{listings}
\usepackage{setspace} 
\usepackage{blindtext}
\usepackage{epstopdf}
\usepackage{afterpage}
\usepackage{parskip}
\usepackage{booktabs}
\usepackage{stackrel}
\usepackage{cite}
\usepackage{fancybox, calc}
\usepackage{verbatim}
\usepackage{cleveref} 
\usepackage[usenames,dvipsnames,svgnames]{xcolor}
\usepackage{tikz}
\usetikzlibrary{shapes,arrows, calc}
\usetikzlibrary{fit,positioning}
\tikzstyle{every node}=[font=\scriptsize]
\tikzstyle{block} = [draw, rectangle]
\tikzstyle{sum} = [draw, circle, node distance=1cm, inner sep=2pt]
\tikzstyle{input} = [coordinate]
\tikzstyle{output} = [coordinate]
\tikzstyle{pinstyle} = [pin edge={to-,thin,black}]
\definecolor{lightseagreen}{rgb}{0.13, 0.7, 0.67}
\definecolor{lightskyblue}{rgb}{0.53, 0.81, 0.98}
\definecolor{lightsalmonpink}{rgb}{1.0, 0.6, 0.6}
\definecolor{mossgreen}{rgb}{0.68, 0.87, 0.68}
\tikzstyle{chart_node_algorithm} = [rectangle, rounded corners, text width=0.17\textwidth, minimum height=1cm,text centered, fill=lightsalmonpink!10, draw=lightsalmonpink, thick]

\tikzstyle{chart_node_lqr} = [rectangle, rounded corners, text width=0.17\textwidth, minimum height=1cm,text centered, fill=mossgreen!10, draw=mossgreen, thick]

\tikzstyle{chart_node_theory} = [rectangle, rounded corners, text width=0.17\textwidth, minimum height=1cm,text centered, draw=lightskyblue, thick, fill=lightskyblue!10]

\newcommand{\HS}{\mathrm{HS}}

\newcommand{\R}{\mathbb{R}}
\newcommand{\E}{\mathbb{E}}
\newcommand{\cK}{\mathcal K}

\newcommand{\cH}{\mathcal H}

\newcommand{\cJ}{\mathcal J}

\newcommand{\norm}[1]{\lVert #1\rVert}

\newcommand{\uf}{\mathfrak u}
\newcommand{\cU}{\mathcal U}
\newcommand{\cX}{\mathcal X}

\newcommand{\cB}{\mathcal B}

\newcommand{\umpc}{{\mathfrak u}_T^{\textrm{MPC}}}
\newcommand{\mumpc}{{\mu}_T}
\newcommand{\mumpck}{{\mu}_k}
\newcommand{\mpc}{\textrm{MPC}}

\newcommand{\ff}{\mathfrak f}

\newcommand{\cD}{\mathcal{D}}
\newcommand{\bbN}{\mathbb{N}}
\newcommand{\bbR}{\mathbb{R}}

\newcommand{\hnorm}[1]{\left\| #1 \right\|_{\cH}}
\newcommand{\abs}[1]{\left| #1 \right|}

\allowdisplaybreaks

\newcommand{\off}[1]{}

\newcommand{\cL}{\mathcal L}
\newcommand{\cR}{\mathcal R}
\newcommand{\bbZ}{\mathbb Z}
\newcommand{\bbT}{\mathbb T}

\DeclareMathOperator*{\argmin}{arg\,min}

\newtheorem{theorem}{Theorem}[section]
\newtheorem{lemma}[theorem]{Lemma}
\newtheorem{corollary}[theorem]{Corollary}
\newtheorem{proposition}[theorem]{Proposition}
\theoremstyle{plain}
\newtheorem{example}[theorem]{Example}
\newtheorem{remark}[theorem]{Remark}

\newtheorem{assumption}[theorem]{Assumption}

\title{Learning to Control Switching Nonlinear Systems with Koopman Operator Regression\thanks{Preprint. Under review.}}
\author[1]{Edoardo Caldarelli\footnote{Equal contribution.}} 
\author[2]{Oleksii Kachaiev$^{\dagger}$}
\author[2]{Cesare Molinari}
\author[1, 3]{Lorenzo Rosasco}
\date{}
\affil[1]{Istituto Italiano di Tecnologia, Genoa, Italy}
\affil[2]{MaLGa center, DIMA, Università degli Studi di Genova, Genoa, Italy}
\affil[3]{MaLGa center, DIBRIS, Università degli Studi di Genova, Genoa, Italy}
\affil[ ]{Correspondence to: \texttt{edoardo.caldarelli@iit.it}}
\begin{document}
\maketitle
\begin{abstract}
In this work, we consider the identification and control of nonlinear systems with finite action spaces. The unknown dynamics are estimated from finite samples with Koopman operator regression in a reproducing kernel Hilbert space, yielding a linear switching predictive model, the switches governed by the value of the control variable. In order to perform control in closed-loop, the learned dynamics are employed in an infinite-horizon optimal control problem with time-varying stage cost, which is solved by means of model predictive control. In a theoretical analysis, we derive learning rates for the Koopman dynamics approximation. We further quantify, under suitable assumptions, the sub-optimality of the model predictive control strategy, both in the case of exact Koopman dynamics, and in the case of learned ones. Numerical simulations on the Duffing oscillator complement our theoretical findings. 
\end{abstract}
\section{Introduction and related works}
Optimal control is a well established, powerful tool to control dynamical systems in closed-loop \cite{lewis2012optimal}. Given a state space $\cX$, a control space $\cU$, a flow map $\ff:\cX\times\cU\to \cX$, and a dynamical system described, e.g., by a difference equation
\begin{equation}
    x_{t+1} = \ff(x_t, u_t),\label{eq:intro_dynsys}
\end{equation}
optimal control entails solving an optimization problem of the form
\begin{equation}
    \min_{u_0, u_1, \dots}\sum_{t=0}^\infty g(t, x_t, u_t), \text{ subject to } x_{t+1} = \ff(x_t, u_t),\ x_0\in\cX. \label{eq:ocp_intro}
\end{equation}
for some stage cost $g:\mathbb N_0\times \cX \times \cU\to \mathbb R_{\geq 0}$. 

While optimal control for linear systems enjoys a long-lasting history of results and well-developed techniques \cite{kwakernaak1972linear,tsiamis2023statistical}, when the system's dynamics are nonlinear, the solution of problem \eqref{eq:ocp_intro} is challenging. The Koopman operator formalism \cite{koopman1931hamiltonian,mezic2021koopman,brunton2022modern,mezic2026koopman} has emerged as an effective paradigm to obtain a global linearization of the dynamics of interest, by lifting the state to a possibly infinite-dimensional space of functions  $\cH$, named \emph{observable space}. Having introduced a \emph{lifting map} $x\mapsto\psi_x$, we can define the lifted state as
\begin{equation}
    z_t = \psi_{x_t}.
\end{equation}
We can therefore describe how such a lifted state evolves over time, in terms of the Koopman operator. 

When considering controlled dynamical systems, the classical formulation of the Koopman operator needs to be adapted to account for the control variable $u$ \cite{brunton2022modern}. State-of-the-art approaches consider additional linearity assumptions on the impact of the control on the lifted dynamics \cite{korda2018linear,caldarelli2025linear,heeg2026limitations}. Other concurrent works, such as \cite{strasser2024safedmd,worthmann2024data,bold2024data}, restrict the dynamics \eqref{eq:intro_dynsys} to be bi-linear, which transfers to a similar structure in the Koopman model. In this work, we assume the set of available controls to be finite (a common feature, e.g., in power electronics applications \cite{karamanakos2019guidelines}). In this way, the nonlinear dynamics translate into a \emph{family} of systems, one for each value of the control variable. Consequently, we are able to define a corresponding family of Koopman operators $\cK_u:\cH\to\cH$, that evolve the lifted state $z$ over time according to the rule 
\begin{equation}
    z_{t+1} = \cK_{u_t} z_t.\label{eq:lifted_dynamics_intro}
\end{equation}
The \emph{linear switching} system in \eqref{eq:lifted_dynamics_intro} was firstly introduced in the seminal work by Peitz \& Klus \cite{peitz2019koopman}, and is analogous to \eqref{eq:intro_dynsys} under suitable assumptions. Furthermore, it can be related to state-of-the-art models in reinforcement learning with operator world models, see the work by Novelli et al.\ \cite{novelli2024operator}. 

Such lifted dynamics can be used to define an optimal control in the lifted state space, i.e., 
\begin{equation}
        \min_{u_0, u_1, \dots}\sum_{t=0}^\infty \ell(t, z_t, u_t), \text{ subject to } z_{t+1} = \cK_{u_t}z_t,\ z_0 = \psi_{x_0}.\label{eq:lifted_ocp_intro}
\end{equation} 
In this way, we can make predictions leveraging the efficiency of model \eqref{eq:lifted_dynamics_intro}, which entails composing linear operators instead of nonlinear flow maps \cite{peitz2019koopman}. Such an optimal control problem can either be designed directly in the lifted state, or derived from a problem of the form \eqref{eq:ocp_intro} by choosing $\ell(t, z, u) = g(t, x, u)$, with $x$  such that $z=\psi_x$. Lastly, note that the finite nature of the control set may render the closed-loop dynamics ultimately bounded, but not asymptotically stable in the limit \cite{aguilera2013stability}. This in turn may yield an unbounded optimal control cost, even though the resulting closed-loop dynamics may be stable. To overcome this issue, we employ here a \emph{time-varying} stage cost $\ell$  \cite{anand2024optimality,schwenkel2024discount,moldenhauer2026discounted}. 

In the framework described above, we tackle  two main challenges. The first one corresponds to the operators $\cK_u$  being unknown. In this case, a family of \emph{surrogate} models $\widehat \cK_u$ is learned from data, and used in place of the exact dynamics in problem \eqref{eq:lifted_ocp_intro}, yielding a \emph{data-driven} control pipeline \cite{soudbakhsh2023data}. To estimate the linearized Koopman dynamics \eqref{eq:lifted_dynamics_intro} from snapshots of the system evolution over time, one may resort to \emph{Koopman operator regression} algorithms \cite{bevanda2021koopman,nuske2023finite}, i.e., system identification techniques applied to models of the form \eqref{eq:lifted_dynamics_intro}. In this work, we learn the Koopman operators by nonparametric regression in \emph{reproducing kernel Hilbert spaces} (RKHSs) \cite{aronszajn1950theory}, and study its convergence guarantees. RKHSs have been broadly popular in machine learning \cite{scholkopf2002learning}, and have been recently investigated in relationship with Koopman operator regression \cite{das2020koopman,klus2020kernel,bevanda2023koopman,khosravi2023representer,philipp2023error}, together with related theoretical properties \cite{kostic2022learning,kostic2023sharp,kostic2024consistent}.
In the framework described above, we tackle  two main challenges. The first one corresponds to the operators $\cK_u$  being unknown. In this case, a family of \emph{surrogate} models $\widehat \cK_u$ is learned from data, and used in place of the exact dynamics in problem \eqref{eq:lifted_ocp_intro}, yielding a \emph{data-driven} control pipeline \cite{soudbakhsh2023data}. To estimate the linearized Koopman dynamics \eqref{eq:lifted_dynamics_intro} from snapshots of the system evolution over time, one may resort to \emph{Koopman operator regression} algorithms \cite{bevanda2021koopman,nuske2023finite}, i.e., system identification techniques applied to models of the form \eqref{eq:lifted_dynamics_intro}. In this work, we learn the Koopman operators by nonparametric regression in \emph{reproducing kernel Hilbert spaces} (RKHSs) \cite{aronszajn1950theory}, and study its convergence guarantees. RKHSs have been broadly popular in machine learning \cite{scholkopf2002learning}, and have been recently investigated in relationship with Koopman operator regression \cite{das2020koopman,klus2020kernel,bevanda2023koopman,khosravi2023representer,philipp2023error}, together with related theoretical properties \cite{kostic2022learning,kostic2023sharp,kostic2024consistent}.

The second challenge is given by the infinite predictive horizon appearing in problem \eqref{eq:lifted_ocp_intro}, which renders it computationally intractable. To deal with this issue,  we compute an approximate solution of problem \eqref{eq:lifted_ocp_intro} via model predictive control (MPC) \cite{rawlings2020model,peitz2019koopman}. From a theoretical point of view, we are concerned with analyzing: How well the MPC algorithm can solve the intractable problem \eqref{eq:lifted_ocp_intro} when the exact Koopman operators are known, (this can be seen as the limit where infinite data are available);  How the estimation error, which we incur by learning the Koopman models from finite data, impacts such a performance, in line with robust MPC approaches \cite{saltik2018outlook}.

In this paper, we consider {unconstrained} MPC without terminal cost. While the addition of terminal ingredients may support a theoretical analysis through the lenses of recursive feasibility \cite{mayne2013apologia}, the design of meaningful constraints and a terminal penalty is often times challenging (see, e.g., \cite{grune2008infinite}). 
\paragraph*{Contributions} Our contributions are summarized as  follows:
\begin{itemize}
\item We study a linear, Koopman-based predictive model for nonlinear dynamics with finite control spaces. 
    \item We show how such Koopman models can be learned by nonparametric regression in an RKHS and derive learning bounds. 
   Importantly, our analysis does not rely on ergodicity or time-reversibility assumptions on the dynamics of interest.
    \item We show how to leverage  the learned surrogate dynamics in a RKHS to formulate an infinite-horizon, unconstrained optimal control problem with time-varying stage cost, designed to accommodate for the finite nature of the control set. Such an optimal control problem is solved using MPC.
    \item We perform a  theoretical performance analysis, where we derive sub-optimality bounds for the MPC algorithm, both when the exact Koopman model or the surrogate models are used.
    \item We provide an open-source implementation\footnote{Code available at \url{https://github.com/caedoard/switching-koopman-mpc}.} of the algorithms discussed in the paper, and assess their performance on an illustrative numerical example, empirically corroborating our theoretical findings.
\end{itemize}

\paragraph*{Outline} The structure of this paper is as follows: Section \ref{sec:koopman-rev} describes and analyzes the considered system identification technique , namely, Koopman operator regression. Section \ref{sec:optimal_control} describes and investigates the theoretical properties of the MPC algorithm used to control the dynamics of interest in closed-loop. Section \ref{sec:experiments} shows the performance of the proposed MPC algorithm on the Duffing oscillator. Lastly, Section \ref{sec:conclusion} discusses conclusions and future work.

\paragraph*{Notations} $\mathbb R_{\geq 0}$ denotes the set of real numbers greater or equal to 0, and $\mathbb N_0$ denotes the set of natural numbers greater or equal to 0. \( \cX \) denotes a measurable state space endowed with Borel \( \sigma \)-algebra \( \cB(\cX) \). \( \mathcal P (\cX) \) is a space of probability measures on \( \cX \). For an operator $A$, $\norm{A}$ denotes its operator norm, $\norm{A}_\HS$ denotes its Hilbert-Schmidt norm (when defined). $\mathcal L^2_\rho$ denotes the space of functions \( f: \cX \to \bbR \) square integrable w.r.t.\ the measure $\rho$. We denote by $\cL(\cH)$ the set of bounded linear operators  on $\cH$, and we let $\HS(\cH)$ the set of Hilbert-Schmidt operators  on $\cH$. Lastly, we denote as $a\wedge b$ the maximum between $a$ and $b$.

\section{Koopman system identification}
\label{sec:koopman-rev}
Let $\cX$ be a state space, and  $\cU$  a control space. We consider a nonlinear dynamical system with a finite control space, i.e., $|\cU|<\infty$. Specifically, we denote by $\ff: \cX \times \cU \to \cX$ the flow map of such a dynamical system, i.e., for $x_0\in\cX$, $t \geq 0$, $u_t\in \cU$, we have
\begin{align}
    x_{t+1} &= \ff(x_t, u_t)\label{eq:dynsys_xu}\\
    &= \ff_{u_t}(x_t).\nonumber
\end{align}
Note that $\ff$ is a \emph{nonlinear switching flow map}: The switches happen at each time step among systems in the \emph{finite family} $\{\ff_u\}_{u\in\cU}$, and are governed by the control $u_t$.

The flow map $\ff$ may be unknown, and can therefore be approximated from some realizations of the system's evolution over time. Learning $\ff$ directly from data would yield a \emph{nonlinear switching data-driven} system. On the other hand, the Koopman operator formalism is a way to obtain \emph{linear} switching dynamics. This reduces the problem of making predictions of the system's evolution to the composition of linear operators \cite{peitz2019koopman}, among the multiple benefits of having globally linear models \cite{brunton2022modern,korda2018linear,mezic2021koopman,strasser2026overview}. The Koopman operator formulation relies on transforming the state of the system to a possibly infinite-dimensional space of design choice. In this work, we opt for RKHSs, as detailed in the following.

\subsection{Koopman lifting in an RKHS}
For a state space $\cX$, an RKHS $(\cH, \langle \cdot, \cdot \rangle_\cH)$ is a Hilbert space of scalar functions on $\cX$ for which there exists a $k:\cX \times \cX \rightarrow \R$, the \emph{reproducing kernel}, such that, for {any $x \in \cX$ and $g \in \cH$, it holds} $k(x, \cdot)\in \cH$ and $g(x) = \langle g, k(x, \cdot)\rangle_{\cH}$ \cite{aronszajn1950theory}.
Every $k$ is a positive definite (p.d.) kernel. This means that, given a collection of $p$ points $x_1,\ \dots,\  x_p\in \cX$, the matrix defined as $K_{i,j} = k(x_i, x_j)$ for $i,j \in \{1,\ \dots,\ p\}$ is positive definite, i.e, 
\begin{equation}
\sum_{i, j} c_ic_jK_{ij}\geq 0, \forall c_i, c_j \in \mathbb R.\label{eq:pd_kernel}
\end{equation}
The canonical feature map of $k$ is defined as $x \mapsto \psi_x \coloneqq k(x,\cdot) \in \cH$, so that \( k(x,y)=\langle \psi_x,\psi_y\rangle_\cH,\ \forall  x,y\in\cX \). 

In this work, we assume $k$ to be \emph{strictly p.d.}, meaning that 0 in $\eqref{eq:pd_kernel}$ is attained iff $c_i=c_j=0,\ \forall i, j$ \cite{scholkopf2002learning}.  
It is straightforward to see, that this assumption ensures  that canonical feature map is injective. One example of strictly p.d.\ kernels is given by translation invariant kernels, i.e., kernels depending on the distance between two points in $\cX$, such as the Gaussian kernel with bandwidth $l>0$: 
\begin{equation}
    k(x, x') = e^{-\frac{{\lVert x- x'\rVert}^2}{2l^2}}.\label{eq:gaussian_kernel}
\end{equation}
Another example of kernel is the so-called \emph{random feature} kernel, which approximates a stationary kernel like \eqref{eq:gaussian_kernel} as the inner product of two finite dimensional vectors \cite{rahimi2007random}.

Let $\Psi(\cX) \coloneqq \{\psi_x:x\in\cX\}$ be the image of the state space according to the canonical feature map. Intuitively, we can interpret the feature map as an embedding, or lifting, of the state space in a high dimensional space. Then $\Psi(\cX)$ is the set of embedded, or lifted, states.  

\paragraph*{Lifted dynamics}Given the lifted state set $\Psi(\cX)$, we can proceed to construct a suitable operator describing how the embedded states evolve over time, according to the Koopman paradigm. Fix \( u\in\cU \). By the injectivity of $\psi$, we can  define the  pre-Koopman operator $K_u:\Psi(\cX)\to\Psi(\cX)$ as
\begin{equation}
K_u\psi_x \coloneqq \psi_{\ff_u(x)}, \quad \forall x\in\cX.
\label{eq:prekoopman}
\end{equation}
Note that \eqref{eq:prekoopman} introduces the notion of \emph{lifted dynamics} in $\Psi(\cX)$, governed by the operator $K_u$. However, $\Psi(\cX)$ is a set, and lacks structure, hindering the theoretical tractability of $K_u$. For this reason, we extend the operator $K_u$ to a linear operator $\cK_u: \cH\to\cH$. This allows to reformulate the system's dynamics in the space $\cH$, via the rule \eqref{eq:prekoopman}. Given an initial state $x_0\in\cX$, we can define the lifted initial state
\begin{equation*}
z_0 \coloneqq \psi_{x_0}.
\end{equation*}
Then, for each $t\ge 0$,
\begin{equation}
z_{t+1} = \cK_{u_t} z_t.
\label{eq:exact-lifted-dynamics}
\end{equation}
This system representation is a \emph{linear switching system} on $\cH$, it is possibly infinite-dimensional, and alleviates the burden of relying on the nonlinear flow map $\ff$. Note that $z_t=\psi_{x_t}$, according to \eqref{eq:prekoopman}. For notational convenience, we will refer to model \eqref{eq:exact-lifted-dynamics} by means of the following function $f:\cH\times \cU \to \cH$,
\begin{equation*}
    f(z, u) \coloneqq \cK_uz.
\end{equation*}
\paragraph*{Extension to whole RKHS}
To construct the extension $\cK_u$, we proceed as follows. First, we define $\cH_0 \coloneqq \operatorname{span} \left\{ \Psi(\cX) \right\}$. By definition, $\cH_0$ is a linear subspace of $\cH$, inheriting the inner product $\langle a, b\rangle_{\cH_0} = \langle a, b\rangle_{\cH}$. Moreover, $\cH_0$ is dense in $\cH$ \cite{aronszajn1950theory}, i.e., $\overline \cH_0=\cH$.
For a strictly p.d.\ kernel, there exists a unique linear extension of $K_u:\Psi(\cX)\to\Psi(\cX)$ to an operator $\widetilde K_u:\cH_0 \to \cH$, as shown in the following proposition. 
\begin{proposition}
\label{lem:linear-extension-H0}
Assume the kernel \( k \) is strictly p.d.
Fix $u \in \cU$. Let $h \in \cH_0$ and let $\{(x_i,\alpha_i)\}_{i=1}^m\subset \cX\times\bbR$ such that $h = \sum_{i=1}^m \alpha_i \psi_{x_i}$. Define 
\begin{equation}
\widetilde K_u\left(h\right)
= \sum_{i=1}^m \alpha_i \psi_{\ff_u(x_i)}.
\label{eq:linear-extension-definition}
\end{equation}
Then $\widetilde K_u$ is a well-defined operator from $\cH_0$ to $\cH$, meaning that the right-hand side is independent of the chosen representation of $h$. Moreover, it is the unique linear operator which extends $K_u:\Psi(\cX)\to\cH$.
\end{proposition}
The proof of this result is reported in Appendix \ref{app:proof-linear-extension}.

In order to extend $\widetilde K_u$ to an operator $\cH\to\cH$, we make the following boundedness assumption.

\begin{assumption}[Uniform boundedness on $\cH_0$]
\label{ass:prekoopman-uniformly-bounded}
For each $u\in\cU$ there exists $R_u<\infty$ such that
\begin{equation}
\label{eq:prekoopman-cond}
\|\widetilde K_u h\|_\cH \le R_u \|h\|_{\cH_0},
\quad \forall h\in\cH_0.
\end{equation}
Let
\begin{equation}
R_{*} \coloneqq \max_{u \in \cU} R_u.
\label{eq:r-star}
\end{equation}
\end{assumption}
\begin{example}[Sobolev space observables]
This assumption is satisfied, for example, in the following setup. Consider $\cX = \bbT^d := \R^d / \bbZ^d$ and let \( \cH=H^s(\cX) \), i.e., the space of Sobolev functions on a periodic domain. Suppose that, for each $u\in\cU$, the map \( \ff_u:\cX \to \cX \) is a \( C^\infty \)-diffeomorphism. Then the composition operator \( C_{\ff_u}:H^s(\cX)\to H^s(\cX) \) given by
\begin{equation*}
C_{\ff_u}h:=h\circ \ff_u,
\quad \forall h \in \cH,
\end{equation*}
is bounded (see, e.g., \cite{adams2003sobolev,bourdaud2022introduction}). Namely, \( \|C_{\ff_u}\|_{\cH \to \cH} \leq R_u < \infty \). Therefore its adjoint \( C_{\ff_u}^\ast \) is also bounded. Note that on \( \cH_0 \), we have
\begin{equation*}
\langle h, \widetilde K_u \psi_x \rangle_\cH
= \langle h, \psi_{\ff_u(x)} \rangle_\cH 
= h \circ \ff_u(x)
= \langle C_{\ff_u} h, \psi_x \rangle_\cH
= \langle h, C_{\ff_u}^\ast \psi_x \rangle_\cH,
\quad \forall h \in \cH_0, \; x \in \cX.
\end{equation*}
Thus,
\begin{equation*}
\| \widetilde K_u h \|_{\cH}
= \| C_{\ff_u}^\ast h \|_{\cH}
\leq R_u \| h \|_\cH.
\end{equation*}
Therefore, condition~\eqref{eq:prekoopman-cond} holds.
\end{example}
Given the above  assumption, we can  extend the operator $\widetilde K_u$ to the whole RKHS $\cH$. 
\begin{proposition}
\label{prop:linear-extension-2}
Let $\widetilde K_u: \cH_0 \to \cH$ be defined as in Proposition \ref{lem:linear-extension-H0}. Let Assumption \ref{ass:prekoopman-uniformly-bounded} hold, i.e., $\widetilde K_u$ be bounded. Then, there exists a unique linear extension $\cK_u:\cH\to\cH$ of $\widetilde K_u$. Such an extension is bounded in operator norm with the same constant as $\widetilde K_u$.
\end{proposition}
The proof of this result is reported in Appendix \ref{app:proof-linear-extension-2}.

In the remainder of the paper we will assume Assumption \ref{ass:prekoopman-uniformly-bounded} to hold,  so that by the above result for each $u\in \cU$, there exists a unique linear extension  $\cK_u:\cH\to\cH$ ; bounded by $R_u$.
We add two remarks.

\begin{remark}[Defining Koopman operators]
\label{remark:concurrent_definitions_koopman}
In the system identification literature, Koopman operators on RKHS are often introduced as restriction of a certain operator  $\mathcal A: \cL_\pi^2 \to \cL_\pi^2$ for an invariant measure $\pi$ (e.g., for ergodic dynamics), see, e.g., \cite{kostic2022learning,kostic2023sharp}. Here, we deliberately avoid assuming existence/knowledge of $\pi$ and instead work directly with an RKHS lifting.
\end{remark}

\begin{remark}[Assumptions on the dynamics]
    Note that our lifted model \eqref{eq:exact-lifted-dynamics} does not rely on a specific structure of $\ff$, other than the boundedness Assumption \ref{ass:prekoopman-uniformly-bounded}. The further assumption on the finite nature of the control set allows to construct the switching dynamics \eqref{eq:exact-lifted-dynamics}. This can be contrasted with other state-of-the-art approaches assuming that the system of interest admits a linear time-invariant representation in $\cH$, i.e., a model of the form \begin{equation}
        z_{t+1} = \mathcal Az_t + \cB u_t
    \end{equation} for $\mathcal A$ and $\cB$ being linear operators\cite{mezic2021koopman,caldarelli2025linear,heeg2026limitations}. Besides these works, other approaches assume a bilinear structure of $\ff$, i.e., 
    \begin{equation}
        \ff(x_t, u_t) = g_1(x_t) + g_2(x_t)u_t
    \end{equation} for suitable functions $g_1$ and $g_2$, which transfers to a similar bi-linear structure in the lifted state space \cite{worthmann2024data,bold2024data}.
\end{remark}

\subsection{Data-driven dynamics}
The operators $\cK_u$ previously introduced may be unknown in practice. Here we discuss how to estimate them by sampling the system's state, and solving a suitable learning problem. 

\paragraph*{Data collection} Fix $u\in\cU$ and a design distribution $\rho_u\in\mathcal{P}(\cX)$. We collect data by sampling
\begin{equation}
x_1,\dots,x_n \overset{\mathrm{i.i.d.}}{\sim}\rho_u\label{eq:sampling_rule}
\end{equation}
and seek a linear operator mapping $\psi_{x_i}\mapsto \psi_{\ff_u(x_i)}$, i.e., approximating the linear evolution of the lifted state described by \eqref{eq:exact-lifted-dynamics}. The choice of $\rho_u$ determines which regions of $\cX$ (hence which lifted states) are prioritized in identification. In particular, if $\ff_u$ admits an invariant measure and one aims for long-horizon prediction under control $u$, choosing $\rho_u$ close to that invariant distribution is a natural option. In absence of this information, choosing $\rho_u$ is part of the design. Common choices include, e.g., uniform sampling from a bounded region of the state space (cf.\ \cite{mezic2021koopman} and Section \ref{sec:experiments}).

\paragraph*{Regression problem} For each $u\in \cU$, once we sampled $n$ state values, we can define $\widehat \cK_u$ as the unique solution of the following regularized operator-valued nonparametric regression problem \cite{kostic2022learning,kostic2023sharp}:
\begin{equation}
\label{eq:emp-reg-risk}
\widehat \cK_u \coloneqq
\argmin_{W \in \HS(\cH)}
\left\{
\frac{1}{n} \sum_{i = 1}^n \bigl\| \psi_{\ff_u(x_i)} - W \psi_{x_i} \bigr\|_\cH^2 + \gamma_u \norm{W}_{\HS}^2
\right\},
\quad \gamma_u>0.
\end{equation}
Although $\cH$ may be infinite-dimensional,~\eqref{eq:emp-reg-risk} admits a closed-form solution via the representer theorem for a certain vector-valued kernel ridge regression \cite{micchelli2005learning,grunewalder2012conditional}, as shown in the following proposition.
\begin{proposition}[Closed-form solution of \eqref{eq:emp-reg-risk}]
\label{prop:closed_form_sol}
    For a given p.d.\ kernel $k$, a control $u\in\cU$, $n\in\mathbb N$, a dataset of states $\{x_1, \dots, x_n\}$ sampled according to \eqref{eq:sampling_rule}, and a state $x\in\cX$,
    define the following matrices
    \begin{equation*}
        H_{nn} \in \mathbb R^{n\times n},\ (H_{nn})_{i, j} \coloneqq k(x_i, x_j) \quad \text{and}\quad  H_{nx} \in \mathbb R^n,\ (H_{nx})_{i} \coloneqq k(x_i, x).
    \end{equation*}
    For 
    \begin{equation*}
        a \in \mathbb R^{n},\ a \coloneqq (H_{nn} + n\gamma_uI)^{-1} H_{nx},
    \end{equation*}
    we have that
    \begin{equation*}
        \widehat \cK_u \psi_x = \sum_{i=1}^n a_i \psi_{\ff_u(x_i)}.
    \end{equation*}
\end{proposition}
    The proof of this result is given in Appendix \ref{app:closed_form_sol}.

Note that even though the dynamics of interest are deterministic, we choose to use of Tikhonov regularization in \eqref{eq:emp-reg-risk} to prevent the calculation of the solution from becoming numerically unstable \cite{giannakis2023learning}. Moreover, in our experiments in Section \ref{sec:experiments}, we will use the approximate, finite-dimensional feature map given by random features \cite{rahimi2007random}.

\paragraph*{Lifted surrogate dynamics} Analogously to the exact case, given the model estimates $\{\widehat\cK_u\}_{u\in\cU}$ and an initial state $x_0$, we define the corresponding lifted dynamics by
\begin{equation*}
\hat z_0 \coloneqq \psi_{x_0},
\end{equation*}
and, for all $t\geq 0$,
\begin{equation}
\hat z_{t+1} = \widehat \cK_{u_t}\hat z_t .
\label{eq:surrogate-lifted-dynamics}
\end{equation}
These linear switching dynamics act as a \emph{surrogate} for the model \eqref{eq:exact-lifted-dynamics}, and can be used to perform approximate predictions of the true system's evolution over time. Similarly to the exact case, for notational convenience, we introduce the following function to refer to model \eqref{eq:surrogate-lifted-dynamics}, $\hat f:\cH \times \cU\to\cH$:
\begin{equation*}
    \hat f(z, u) \coloneqq \widehat \cK_uz.
\end{equation*}
Before moving forward discussing how the above surrogate models can be used for control, we first study how well we can expect them to estimate the exact models, provided finite data.

\subsection{Learning guarantees} 
\label{subsec:koopman_rates}
In this section, we study the estimation error incurred in estimating  $f$ in \eqref{eq:exact-lifted-dynamics} by the surrogate data driven model $\hat f$ in \eqref{eq:surrogate-lifted-dynamics}. Towards this end, we need several assumptions. 

The first assumption is standard in the context of statistical learning theory with kernels, see e.g., \cite{kostic2023sharp}.
\begin{assumption}[Uniformly bounded kernel]
\label{ass:bounded-kernel}
There exists $\kappa<\infty$ such that $\|\psi_x\|_\cH\le \kappa$ for all $x\in\cX$.
\end{assumption}
This assumption is fulfilled, e.g., by the Gaussian kernel introduced in \eqref{eq:gaussian_kernel} and more general by translation invariant kernels.

The next assumption strengthens the property of the Koopman operators by requiring a bounded Hilbert–Schmidt norm rather than just a bounded operator norm.

\begin{assumption}[Well-specified model]
\label{ass:koopman-well-specified}
For all $u\in\cU$, $\lVert\cK_u\rVert_\HS <\infty$.
\end{assumption}
We assume a further condition, which can be interpreted as quantifying the alignment of the target operator \( \cK_u \) with the kernel-induced geometry, and the choice of design distributions \( \rho_u \).
\begin{assumption}[Koopman source condition]
\label{ass:koopman-source-condition}
For every \( u \in \cU \), define the operator \( \Sigma_u \in \HS(\cH) \) by
\begin{equation*}
\Sigma_u \coloneqq \int \psi_x \otimes \psi_x \, \rho_u(dx).
\end{equation*}
Then, there exists $r_u \in (1/2,1]$ such that
\begin{equation*}
\norm{\cK_u\,\Sigma_u^{1/2-r_u}}_\HS =: G_u < \infty.
\end{equation*}
\end{assumption}
Note that a larger \( r_u \) corresponds to smoother/easier identification and therefore yields faster rates, as we will show in the following. 
The two latter are assumptions are inspired by analogous conditions in supervised learning \cite{meunier2024optimal}, and have been previously used to study learning Koopman operators \cite{kostic2022learning,kostic2023sharp}.

Under the above assumptions, we can state our main result about  the estimation of the exact model $f$ by the data driven surrogate $\widehat f$  learned from  finitely many samples.
\begin{theorem}[Koopman identification rate]
\label{thm:koopman-sample-rate}
Let Assumptions~\ref{ass:bounded-kernel}, \ref{ass:koopman-well-specified}, and \ref{ass:koopman-source-condition} hold. Define \( r \coloneqq \min_{u\in\cU} r_u \). For every \( u \in \cU \) set
\[
\gamma_u = c_u \, n^{-1/(2r_u+1)}
\quad \text{for some } c_u>0.
\]
Let the family of operators \( \left\{ \widehat \cK_{u} \right\}_{u \in \cU} \) be given by solving~\eqref{eq:emp-reg-risk} for every \( u \in \cU \). Then, for
any $\delta\in(0,1)$, with probability at least $1-\delta$,
\begin{equation*}
\bigl\|\hat f(\psi_x,u)-f(\psi_x,u)\bigr\|_\cH
\;\lesssim\;
\log \! \left(2|\cU| / \delta\right) \cdot \kappa n^{-\frac{2r-1}{4r+2}}, \ \forall x \in \cX,\ \forall u \in \cU.
\end{equation*}
\end{theorem}
    The proof of this result can be found in Appendix~\ref{app:proof_of_learning_rate}.

Our bound shows that, in the fastest regime, the one-step-ahead prediction error of the approximated Koopman model scales at a rate in $\mathcal O(n^{-1/6})$. This bound  matches analogous results in supervised learning \cite{caponnetto2007optimal}, hence suggesting that the obtained estimates are sharp. Compared to other analyses performed in Koopman operator learning with RKHSs, our result is derived from different assumptions on the system dynamics, e.g., we do not assume stationary \cite[Theorem 3]{kostic2022learning} (cf.\ Remark \ref{remark:concurrent_definitions_koopman}). Other quantitative results related to Koopman operator learning can be found in \cite{korda2018convergence,mezic2022numerical,bevanda2023koopman,bevanda2026nonparametric,nuske2023finite}, see also the survey \cite[Section 2.4]{strasser2026overview}. Note that, as it can be expected since we perform a union bound, our bound grows logarithmically in the size of the control set $\cU$. To conclude, note that Theorem~\ref{thm:koopman-sample-rate} implies that the one-step mismatch vanishes as $n\to\infty$, meaning that, under suitable assumptions, the model \eqref{eq:exact-lifted-dynamics} can be learned to an arbitrary level of accuracy, by sampling enough points. 

Provided with the above estimates we go back discussing how the Koopman operators and their data driven surrogate models can be used for control. We will see that such estimates will be needed to assess the robustness of the proposed approach. 

\section{Controller design and performance}
\label{sec:optimal_control}
As mentioned before, a linear switching model \eqref{eq:exact-lifted-dynamics} based on the exact Koopman operators $\cK_u$'s,  can be used to define an infinite-horizon optimal control problem, with the ultimate goal of controlling the dynamical system in closed-loop. 

As the optimal control problem is intractable, we choose to solve it approximately with MPC. In this section, we will formalize such a receding-horizon control strategy, and theoretically analyze its performance w.r.t.\ the infinite-horizon controller. Furthermore, we will consider an approximation of such an MPC algorithm, in which the inaccessible exact dynamics \eqref{eq:exact-lifted-dynamics} are replaced by \eqref{eq:surrogate-lifted-dynamics} to make predictions. In this case, we show that the use of MPC causes a degradation of the controller's performance that is related to the approximation error from Theorem \ref{thm:koopman-sample-rate}.

\subsection{Optimal control problem} 
In order to control the dynamics of interest in closed-loop, we formulate an optimal control problem. To start with, let us introduce the stage cost $\ell \colon \mathbb N_0\times\mathcal H\times \mathcal U\to \R_{\geq 0}$. Note that our stage cost $\ell$ explicitly depends on time (e.g., in the form of a discount factor \cite{anand2024optimality}). This, in turn, allows to filter out oscillations at infinity, which may happen due to our control set being finite \cite{aguilera2013stability}.

For $T\in\mathbb N_0\cup \{\infty\}$, we introduce the  performance index $\cJ_T: \mathbb N_0\times \cH \times \cU^T\to \mathbb R_{\geq 0}$: 
\begin{equation}
    \mathcal J_T(t, z, \uf)\coloneqq\sum_{j=t}^{T+t-1}\ell(j, z_j, u_j),\label{eq:def_perf_indx}
\end{equation}
where $z_{j+1} = f(z_j, u_j), u_j = \{\uf\}_j$, $z_t=z$. 

We can further define the value function $V_T(t, z):\mathbb N_0\times \cH \to \mathbb R_{\geq 0}$ associated to \eqref{eq:def_perf_indx} as 
\begin{equation}
    V_T(t, z) \coloneqq \inf_{\uf\in \mathcal U^T}\mathcal J_T(t, z, \uf).
\label{eq:def-vT}
\end{equation}
If $T<\infty$, such an infimum is computed over a finite set and is therefore attained. If $T=\infty$, we assume that such an infimum is attained, as done, e.g., in \cite{grune2008infinite}. We therefore aim at solving the following optimal control a, for $x_0\in\cX$, $z_0 = \psi_{x_0} $:
\begin{equation}
    	 \min_{\uf\in \mathcal U^\infty}\mathcal J_\infty(0, z_0, \uf).\label{eq:prob_we_solve}
\end{equation}
\begin{remark}[Lifted vs.\ classical optimal control]
\label{remark:originality_of_ocp}
Note that our optimal control problem \eqref{eq:prob_we_solve} is directly defined in the lifted state space, rather than the original one. In contrast, classical optimal control aims at solving, problem \eqref{eq:ocp_intro} subject to the nonlinear dynamics $x_{t+1} = \mathfrak f(x_t, u_t),\ x_0\in\cX$. 
Nonetheless, given a problem of the form \eqref{eq:ocp_intro}, for $t\in\mathbb N_0$, and $u_t\in\cU$, we can transform it to a problem in the lifted state space by choosing
\begin{equation*}
    \ell(t, z_t, u_t) = g(t, x_t, u_t),
\end{equation*}
where $x_t$ is such that $z_t = \psi_{x_t}$, and constraining with the lifted dynamics $z_0 = \psi_{x_0}$, $z_{t+1} = f(z_t, u_t)$. 
Note that this approach requires inverting the feature map $\psi$.
\end{remark}
\begin{remark}[Inversion of the lifting map]
If we define the stage cost $\ell$ directly in lifted state space, we do not necessarily need to use an inverse map $\cH\to\cX$, in contrast to other Koopman-based optimal control strategies that require an inversion of $\psi$ \cite{korda2018linear,caldarelli2025linear} (cf.\ Remark \ref{remark:originality_of_ocp} and Section \ref{sec:experiments}).
\end{remark}
\begin{algorithm}[t]
\caption{Exact MPC algorithm}\label{alg:mpc}
\begin{algorithmic}
\Require Initial state $x_0\in\cX$, switching flow map $f:\cH\times \cU \to \cH$.
\Ensure Receding horizon control sequence $
    \umpc \coloneqq \{\mumpc(0, z_0), \mumpc(1, z_1), \dots \}$.
    \State Initialize $z_0\leftarrow\psi_{x_0}.$
\For{$t\in\mathbb N$}
\State Compute 
	\begin{equation} \uf^{*}_t \in\arg \min_{\uf\in \mathcal U^T}\mathcal J_T(t, z_t, \uf).\label{eq:mpc_ocp}
    \end{equation}
\State $\mu_T(t, z_t )\leftarrow\{\uf^{*}_t\}_0$. 
\State $z_{t+1}\leftarrow f(z_t, \mu_T(t, z_t ))$.
\EndFor
\end{algorithmic}
\end{algorithm}
\subsection{Exact MPC algorithm} Note that problem \eqref{eq:prob_we_solve} is in general intractable, since it entails a minimization over an infinitely long sequence of controls, i.e., switches among dynamical models. To deal with this issue, we solve problem \eqref{eq:prob_we_solve} with MPC. In this subsection, we employ the exact dynamics \eqref{eq:exact-lifted-dynamics} to make predictions, and repeatedly perform the receding-horizon steps reported, for completeness, in Algorithm \ref{alg:mpc}. 

By applying Algorithm \ref{alg:mpc}, we obtain an infinitely long sequence of controls $\umpc\in \cU^\infty$, depending on the implicit feedback law $\mu_T:\mathbb N_0\times \cH \to \cU$ defined in Algorithm \ref{alg:mpc}:
\begin{equation}
    \umpc \coloneqq \{\mumpc(0, z_0), \mumpc(1, z_1), \dots \},\label{eq:mpc_ctrl_seq}
\end{equation}
where the states are visited by playing the receding horizon controls. 
\subsubsection{Performance bound}
In order to assess the performance of Algorithm \ref{alg:mpc}, we
compare the related cost $\cJ_\infty(0, z_0, \umpc)$, with the cost of an optimal control sequence $\uf^* \in \arg\min_{\uf\in \mathcal U^\infty}\mathcal \cJ_\infty(0, z_0, \uf)$. Recall that, by definition, $J_\infty(0, z_0, \uf^*) = V_\infty(0, z_0)$.
This comparison is a standard way to assess the performance of MPC control strategies, see, e.g., \cite{grune2008infinite,grune2015robustness}, and allows to quantify the so-called \emph{sub-optimality gap} we incur when using MPC to approximate an optimal control law.

To derive error bounds, we make the following controllability assumption (cf.\ \cite[Proposition 4.7]{grune2008infinite}).
\begin{assumption}[Stage cost controllability]
\label{ass:stage_cost_ctrl}
    Let $\ell: \mathbb N_0 \times \cH \times \cU\to \mathbb R_{\geq 0}$ be the stage cost. Then, $\exists \lambda \in \ell_1$ such that, $\forall t\in\mathbb N_0$, $\forall x\in\cX$ and $ \forall u \in \cU$, $\exists \bar u \in \cU^\infty$ such that, for $j > t$, 
    \begin{equation*}
        \ell(j, \bar z_j, \bar u_j)\leq \lambda_j\ell(t, z, u).
    \end{equation*}
    Here, $\bar z_{j+1} = f(\bar z_j, \bar u_j)$, $\bar z_t = z=\psi_x$, $\bar u_j = \{\bar u\}_j$, $\lambda_j = \{\lambda\}_j$.
\end{assumption}
The above assumption requires the existence of a control sequence that reduces the value of the stage cost at a sufficiently fast rate. Note that, compared to similar assumptions appearing in the literature, see e.g. \cite[Proposition 4.7]{grune2008infinite}, our Assumption \ref{ass:stage_cost_ctrl} couples the effect of time on the stage cost (e.g., in the form of a discount factor) with the classic controllability of the stage cost to $0$. This is a natural consequence of using a time-varying stage cost, to render the optimal control problem tractable, in presence of finite actions \cite{aguilera2013stability}.

We are now ready to discuss the sub-optimality gap caused by the MPC algorithm.
\begin{theorem}[Sub-optimality of Algorithm \ref{alg:mpc}]
\label{thm:suboptimality_gap}
    Let Assumption \ref{ass:stage_cost_ctrl} hold, let $C\coloneqq \sum_{j=0}^{\infty}\lambda_j$ and let $\umpc$ be defined as in \eqref{eq:mpc_ctrl_seq}.  Assume the length of the predictive horizon $T$ fulfills the following inequality:
\begin{equation*}
        T\geq \frac{2\ln C}{\ln(C) - \ln(C-1)}.
\end{equation*}
Then, for $\alpha\coloneqq 1 - \frac{{(C - 1)}^T}{C^{T-2}}$, and $x_0\in\cX$, $z_0=\psi_{x_0}$, we have that
    \begin{equation*}
        \frac{\cJ_\infty(0, z_0, \umpc)- V_\infty(0, z_0)}{V_\infty(0, z_0)}\leq \frac{1-\alpha}\alpha\leq \frac{\left(\frac{C-1}{C}\right)^T}{C^2-1}.
    \end{equation*}
\end{theorem}
    The proof is reported in Appendix \ref{sec:proof_exact_mpc_suboptimality}.

Theorem \ref{thm:suboptimality_gap} shows that the sub-optimality gap induced by MPC decays exponentially fast in the size of the predictive horizon $T$, justifying the truncation of the horizon in \eqref{eq:prob_we_solve} to a finite value, and the use of Algorithm \ref{alg:mpc}. As a consequence, depending on $C$, the predictive horizon $T$ (and consequently the number of iterations needed to solve \eqref{eq:mpc_ocp}) can be shrunk, to the point of allowing a brute-force (hence exact) solution of problem \eqref{eq:mpc_ocp}. This will be confirmed in our simulations in Section \ref{sec:experiments}.
\begin{remark}[Comparison with \cite{grune2008infinite}]
The resulting sub-optimality rate matches the one firstly derived in \cite{grune2008infinite} in the context of MPC with time-invariant stage cost. In this paper, we provide an original proof for this result, showing that the rate is indeed optimal under Assumption \ref{ass:stage_cost_ctrl}.
   \end{remark} 
\begin{algorithm}[t]
\caption{Inexact MPC algorithm}\label{alg:mpc_approx}
\begin{algorithmic}
\Require Initial state $x_0\in\cX$, inexact switching flow map $\hat f:\cH\times \cU \to \cH$, exact switching flow map $ f:\cH\times \cU \to \cH$.
\Ensure Receding horizon control sequence $
    \hat\uf^\mpc_T \coloneqq \{\hat \mu_T(0, z_0), \hat \mu_T(1, z_1), \dots \}$.
    \State Initialize $z_0\leftarrow\psi_{x_0}.$
\For{$t\in\mathbb N$}
\State Compute 
	\begin{equation} \hat\uf^{*}_t \in\arg \min_{\uf\in \mathcal U^T}\widehat{\mathcal J}_T(t, z_t, \uf).\label{eq:mpc_ocp_approx}
    \end{equation}
\State $\hat \mu_T(t, z_t )\leftarrow\{ \hat\uf^{*}_t\}_0.$ 
\State  $z_{t+1}\leftarrow f(z_t, \hat \mu_T(t, z_t ))$.
\EndFor
\end{algorithmic}
\end{algorithm}
\subsection{Approximate MPC algorithm}
Next,  we can  design an MPC strategy based on $\hat f$ from \eqref{eq:surrogate-lifted-dynamics}, which can be applied when the model $f$ from \eqref{eq:exact-lifted-dynamics} is  unknown. Specifically, we introduce the  cost $\widehat\cJ_T: \mathbb N_0\times \cH \times \cU^T\to \mathbb R_{\geq 0}$:
\begin{equation*}
    \widehat{\mathcal J}_T(t, z, \uf)=\sum_{j=t}^{T+t-1}\ell(j, \hat z_j, u_j),\label{eq:def_perf_indx_approx}
\end{equation*}
	 where $\hat z_{j+1} = \hat f(\hat z_j, u_j), u_j = \{\uf\}_j$, $\hat z_t=z$. We  employ such a cost in an approximate MPC, as detailed in Algorithm \ref{alg:mpc_approx}. Note tha, such an algorithm uses the inexact model $\hat f$ to make predictions, but gets feedback from the exact model $f$. 
In this way, we obtain an infinitely long control sequence $\hat \uf^\mpc_T\in \cU^\infty$.

\subsubsection{Performance bound} Finally, we are interested in quantifying the sub-optimality gap we incur when employing the MPC strategy from Algorithm \ref{alg:mpc_approx}. In order to do this, we need additional assumptions.

The  first  assumption concerns the effect of time on the stage cost $\ell$.
\begin{assumption}[Stage cost decay]
\label{ass:decay_stage_cost}
There exists a sequence $\left\{\beta_t > 0 \right\}_{t\in\bbN_0}$ with $\beta_0=1$ and $B \coloneqq \sum_{t=0}^{\infty} \beta_t < \infty$ such that
\begin{equation*}
\label{eq:t_stage_cost_upper_bound}
\ell\bigl(t,\psi_x,u\bigr) \le \beta_t \,\ell\bigl(0,\psi_x,u\bigr)
\quad \forall t \in \bbN_0, \, x \in \cX, \, u \in \cU.
\end{equation*}
\end{assumption}
As discussed in Appendix \ref{appendix:lq_cost_rkhs}, this assumption is fulfilled, e.g., when using a discounting strategy.

Further, we make an assumption relating the value function $V_T$ to the approximate feedback law $\hat \mu$ and the surrogate dynamics from \eqref{eq:surrogate-lifted-dynamics}.
\begin{assumption}[Lyapunov decrease]
\label{ass:value_function_decay_robust}
There exist \( \widehat \alpha_T\geq0\) and $\widehat\xi_T {\geq 0}$ such that for \( \forall x \in \cX \) and \( \forall t \in \bbN_0 \), denoting \( z \coloneqq \psi_{x} \),
\begin{align*}
V_T\bigl(0, \hat f(z, \hat\mu_T(t,z))\bigr) - V_T(0,z) &\leq - \widehat\alpha_T\,\ell\bigl(0, z, \hat\mu_T(t, z)\bigr) \ + \ \widehat\xi_T.
\end{align*}
\end{assumption}
This assumption resembles a practical Lyapunov condition \cite{limon2009input} relating the inexact dynamics, used to compute the inexact feedback $\hat \mu$, to the exact dynamics, based on which $V_T$ is computed. Assumptions of this kind are common tools in the analysis of inexact MPC algorithms, see, e.g., \cite[Proposition 2.4]{grune2009analysis}, \cite[Theorem 5.8]{grune2015robustness}.

Lastly, we make an assumption on the impact of the one-step-ahead prediction error between models \eqref{eq:exact-lifted-dynamics} and \eqref{eq:surrogate-lifted-dynamics} to the value function, computed over a window of size $T$.
\begin{assumption}[Predictive error propagation]
\label{ass:value_function_lipschitz_cont}
Fix \( T \in \bbN \). There exist a continuous non-decreasing  function \( \rho_T:\bbR_{\geq 0}\to\bbR_{\geq 0} \) with \( \rho_T(0)=0 \) such that, \( \forall x \in \cX \) and \( \forall t \in \bbN_0 \), denoting \( z \coloneqq \psi_x \), 
\begin{align*}
V_T\bigl(0, f(z, \hat\mu_T(t,z))\bigr) - V_T\bigl(0, \hat f(z, \hat\mu_T(t,z))\bigr)  \leq \rho_T\!\left( \norm{f(z,\hat\mu_T(t,z)) - \hat f(z,\hat\mu_T(t,z))}_\cH \right).
\end{align*}
\end{assumption}
This assumption implies that the impact on the value function $V_T$ of making a wrong prediction is regular enough, the regularity being described by the $\rho_T$ function. This resembles similar assumptions appearing in the literature, see, e.g., the uniform continuity conditions appearing in \cite[Section 6.2]{grune2015robustness}. Assumption~\ref{ass:value_function_lipschitz_cont} is stated in a general form. We further provide a sufficient condition under which it holds, which may be easier to verify in practice for a given control setup.
\begin{proposition}
\label{prop:VT_lipschitz_rho}
Assume that there exists $\eta \geq 0$ such that, $\forall n\in \mathbb N$ and any dataset of points $x_1, \dots, x_n\in \cX$, denoting the family of solutions of \eqref{eq:emp-reg-risk} for that dataset as $\{\widehat K_u\}_{u\in\cU}$,
\begin{equation*}
    \max_{u\in \cU}\lVert \widehat K_u\rVert \leq \eta.
\end{equation*}
Moreover, let Assumption~\ref{ass:bounded-kernel} hold.
Define the closed ball in RKHS
\begin{equation}
\cD_0 \coloneqq \left\{ h \in \cH: \norm{h}_\cH \leq \kappa \, \eta \right\}.
\label{eq:def-cd-0}
\end{equation}
Furthermore, for \( t \in \bbN \) and $R^*$ as in \eqref{eq:r-star}, define
\begin{equation}
\cD_t \coloneqq \left\{ h \in \cH: \norm{h}_\cH \leq \kappa \, \eta R_*^{\,t} \right\}.
\label{eq:def-cd-t}
\end{equation}
Fix $T \in \bbN$. Assume that, for every $t \in \left\{ 0,1,\dots,T-1\right\} $, there exists $L_t>0$ such that, for every $u \in \cU$ and for every $w,w' \in \cD_t$ 
\begin{equation}
\label{eq:ass_time_dep_lipschitz_ell}
\bigl|\ell(t,w,u)-\ell(t,w',u)\bigr| \le L_t \, \norm{w - w'}_{\cH}.
\end{equation}

Then, for every \( z, z' \in \cD_0 \), for \( R_* \) given in~\eqref{eq:r-star},
\begin{equation}
\label{eq:VT0_lipschitz}
\bigl|V_T(0,z)-V_T(0,z')\bigr|
\le \left(\sum_{t=0}^{T-1} L_{t} \, R_*^{t}\right)\,\norm{z-z'}_{\cH}.
\end{equation}

Consequently, for every $x \in \cX$ and every $u \in \cU$, Assumption~\ref{ass:value_function_lipschitz_cont} holds with
\begin{equation*}
V_T\bigl(0,f(\psi_x,u)\bigr)-V_T\bigl(0,\hat f(\psi_x,u)\bigr)
\le \rho_T\!\left(\norm{f(\psi_x,u)-\hat f(\psi_x,u)}_{\cH}\right),
\quad
\rho_T(r)\coloneqq\left(\sum_{t=0}^{T-1} L_{t} \, R_*^{\,t}\right)\,r.
\end{equation*}
\end{proposition}
The proof of this result is reported in Appendix \ref{app:V_lipschitz_rho}.

We are now ready to state our main sub-optimality bound for Algorithm \ref{alg:mpc_approx}, comparing the cost $\cJ_\infty(0, z_0, \hat \uf^\mpc_T)$ against $V_\infty(0, z_0)$.
\begin{theorem}[Sub-optimality of Algorithm \ref{alg:mpc_approx}]
\label{thm:alpha-perf-nonhomogeneous-v2}
Fix a predictive horizon \( T \ge 2 \). Let the MPC control law $ \hat \mu_T $ and the corresponding control sequence $\hat \uf^\mpc_T$ be computed according to Algorithm \ref{alg:mpc_approx}. 
Let Assumptions \ref{ass:bounded-kernel}, \ref{ass:koopman-well-specified}, \ref{ass:koopman-source-condition} hold. Moreover, 
let Assumptions \ref{ass:decay_stage_cost}, \ref{ass:value_function_decay_robust}, \ref{ass:value_function_lipschitz_cont} on the control problem hold. Assume there exists \( \rho_\infty: \bbR_{\ge 0} \to \bbR_{\ge 0} \) such that for every \( T \in \bbN \), \( \rho_T \le \rho_\infty \). For $n> 0$, $\delta \in (0, 1]$, define \( \varepsilon(n, \delta) < \infty \) to be a scalar such that
\begin{equation}
\label{eq:model-mismatch-bound}
\sup_{x \in \cX} \norm{f\bigl(\psi_x,\hat \mu_T(0, \psi_x)\bigr)- \hat f\bigl(\psi_x,\hat \mu_T(0, \psi_x)\bigr)}_\cH \leq \varepsilon(n, \delta).
\end{equation}
with probability at least $1-\delta$. Fix \( x_0 \in \cX \) and let \( z_0 = \psi_{x_0} \in \cH \). Then the following bound holds with probability at least $1-\delta$:
\[
\cJ_\infty(0, z_0, \hat \uf^\mpc_T) \le \frac{1}{\widehat \alpha_T}\,V_\infty(0, z_0) +\frac{B}{\widehat \alpha_T}\widehat \xi_T+ \frac{B}{\widehat \alpha_T} \, \rho_\infty(\varepsilon(n, \delta)).
\]
\end{theorem}
    The proof of this result is reported in Appendix \ref{appendix:robust_ctrl_results}.

Given a sufficiently large training set of $n$ state values and the surrogate lifted model $\hat f$, the event \eqref{eq:model-mismatch-bound} holds with probability $1-\delta,$ for $\delta \in (0, 1]$, and $\epsilon(n, \delta) \sim \log \! \left(2|\cU| / \delta\right) \cdot \kappa n^{-\frac{2r-1}{4r+2}}$, cf.\ Theorem \ref{thm:koopman-sample-rate}. Consequently, Theorem~\ref{thm:alpha-perf-nonhomogeneous-v2} provides a high-probability bridge between statistical Koopman learning and robust MPC performance. Under the source and well-specifiedness assumptions, the learned lifted model becomes increasingly accurate with more data, and this accuracy propagates quantitatively to the closed-loop cost. Note that the sole term depending exclusively on our finite-sample error rate is $\epsilon(n, \delta)$, but only through the function \(\rho_\infty\). On the other hand, $\widehat \alpha_T$ and $\widehat \xi_T$ capture the interplay between the model accuracy, and the predictive power of the MPC algorithm in terms of the horizon $T$. The closest result to ours is given by \cite[Theorem 8.3]{grune2015robustness}, where similar empirical quantities to our $\widehat \alpha_T$ and $\widehat \xi_T$ appear (cf.\ Lemma 8.1). Similar to ours, their result considers MPC for nonlinear dynamical systems and arbitrary optimal control objectives, differently from other approaches, that obtain robust performance bounds in the case of linear dynamics and quadratic costs \cite{berberich2025overview}. However, their analysis is performed for a finite-dimensional state space, while ours is applicable to infinite dimensions as well. Moreover, the sub-optimality gap in \cite{grune2015robustness} relies on a semi-global practical asymptotic stability result, while our quantification only relies on properties of the value function. 
\section{Numerical simulations}
\label{sec:experiments}
In this section we provide an empirical validation of the performance of Algorithm \ref{alg:mpc_approx}, on the Duffing oscillator benchmark.
Specifically, we consider the forced oscillator described by the differential equations \cite{korda2020optimal,caldarelli2025linear}:
\begin{equation}
\dot x{(1)} = x{(2)}\qquad \dot x{(2)} = -0.5 x{(2)} -x{(1)}(4{x{(1)}}^2-1) + 0.5u.\label{eq:duffing}
\end{equation}
The ultimate control goal is the stabilization of the origin $[0, 0]^T$. The continuous dynamics are discretized with the Runge-Kutta method. 
\paragraph*{Stage cost}  Setting $\lambda = 0.9999$, $r\geq 0$, our stage cost is defined as
\begin{equation}
    \ell(t, \hat z_t, u_t) = \lambda^t (\norm{\hat z_t - \psi_0}_\cH^2+r\, c(u_t)),\label{eq:stage_cost_duffing}
\end{equation}
subject to the dynamics
\begin{equation*}
\hat z_0 = \psi_{x_0}
\quad \text{and} \quad
\hat z_{t+1} = \widehat \cK_{u_t} \hat z_t,
\end{equation*}
where $c(u)$ is the cost associated to the control $u \in \cU$. Note that the use of the stage cost \eqref{eq:stage_cost_duffing} within an RKHS fulfills the assumptions of Theorem \ref{thm:alpha-perf-nonhomogeneous-v2}, see Appendix \ref{appendix:lq_cost_rkhs} for details.

\paragraph*{Training setup} The data-driven Koopman dynamics are obtained by lifting the state with 400 random Fourier features (RFFs) of the Gaussian kernel \eqref{eq:gaussian_kernel} \cite{rahimi2007random}. 
The kernel lengthscale is set to $0.5$. 
The training data are generated by sampling initial conditions from the unit circle, and the corresponding one-step-ahead state values obtained via \eqref{eq:duffing}. {For all $u\in \cU$,  we choose $\rho_u$, see \eqref{eq:sampling_rule}, to be the uniform measure on the closed ball, since we do not have any prior information. The choice of other measures can be further explored as future work (e.g., by means of density estimation over Monte Carlo samples of trajectories).} The values of $u$ defining $\mathcal U$ are fixed as discussed in Section \ref{subsec:symm_ctrl_set} and Section \ref{subsec:asymm_ctrl_set}. 
\paragraph*{Evaluation} 
\begin{figure*}
\begin{subfigure}{.245\linewidth}
    \includegraphics[width=\linewidth]{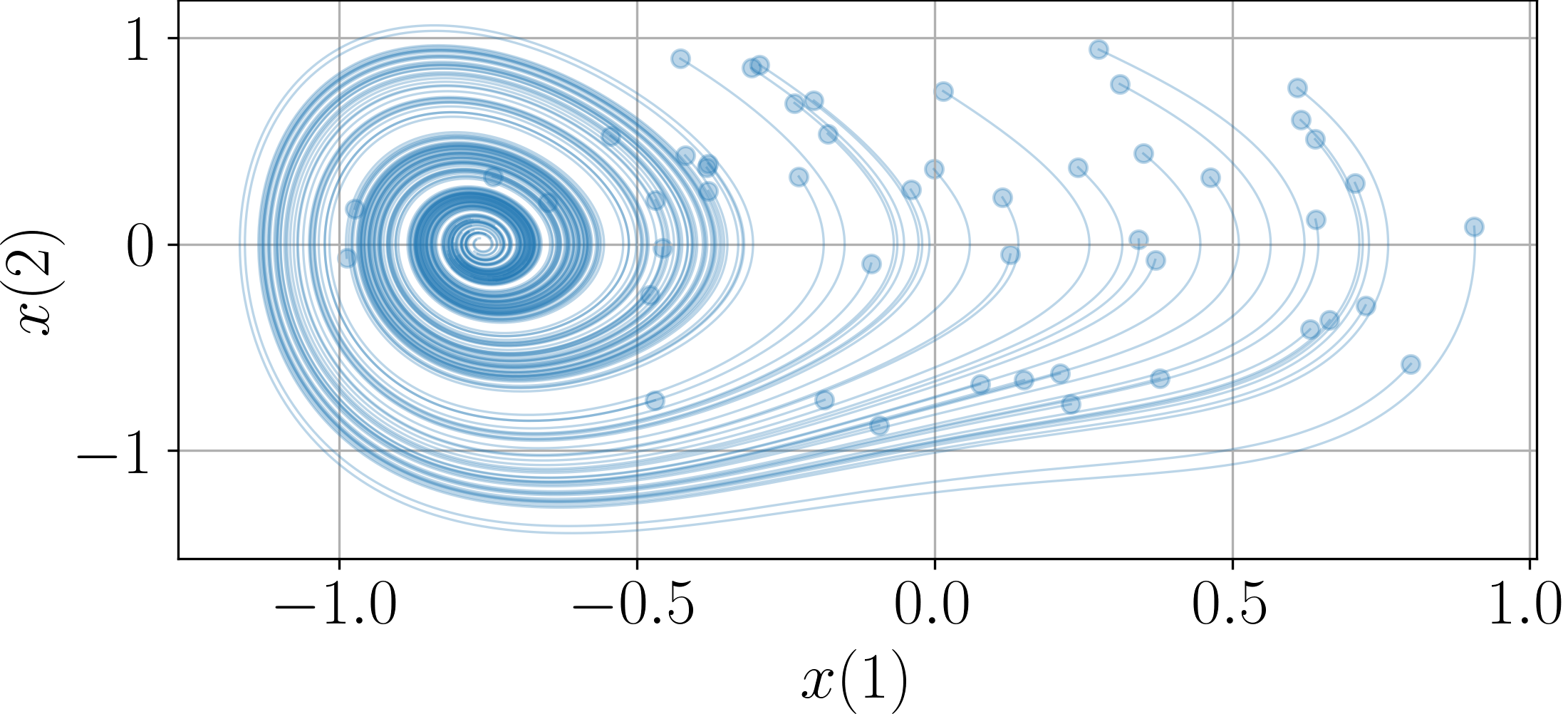}
    \caption{$u=-2$.}
    \label{fig:sample_trajs_-2}
\end{subfigure}
\begin{subfigure}{.245\linewidth}
    \includegraphics[width=\linewidth]{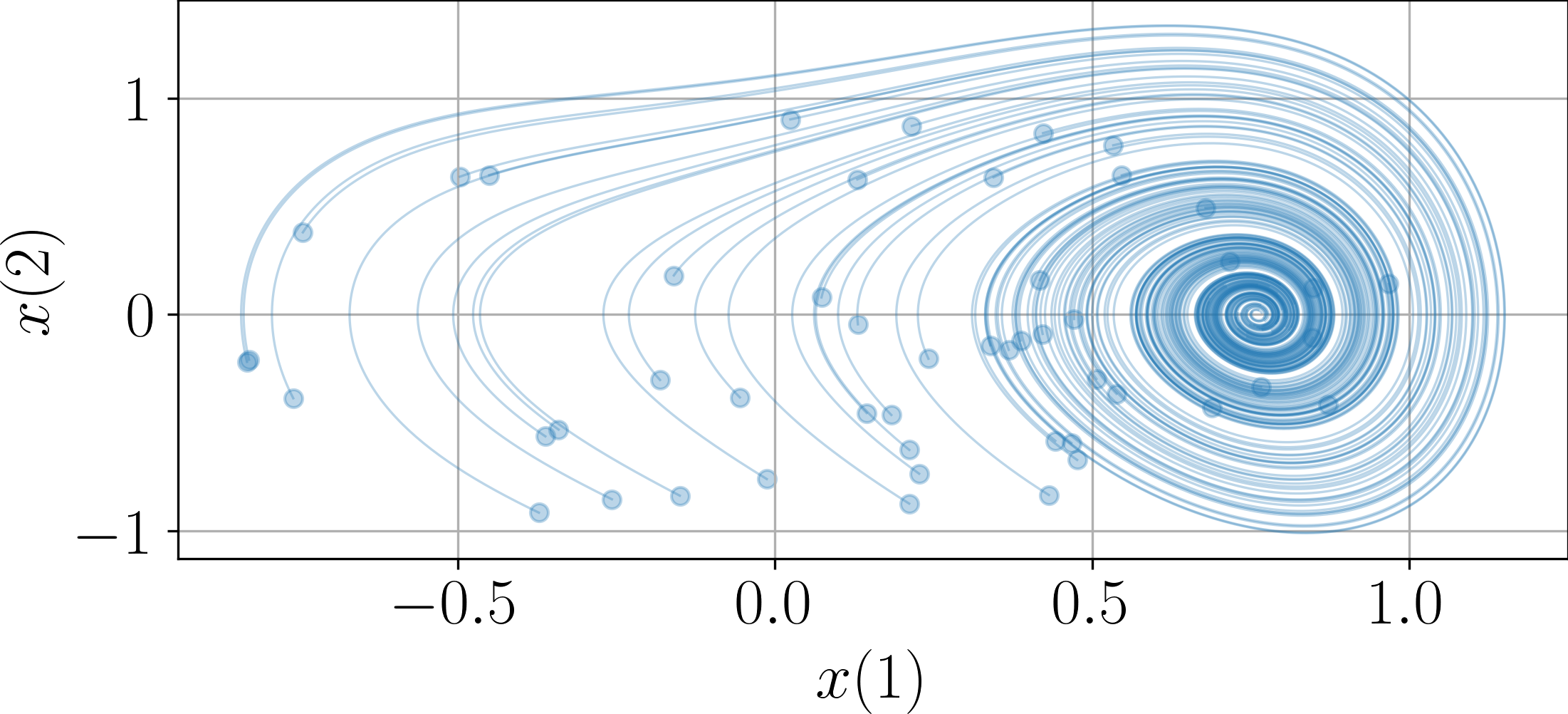}
    \caption{$u=2$.}
        \label{fig:sample_trajs_2}
\end{subfigure}
\begin{subfigure}{.245\linewidth}
    \includegraphics[width=\linewidth]{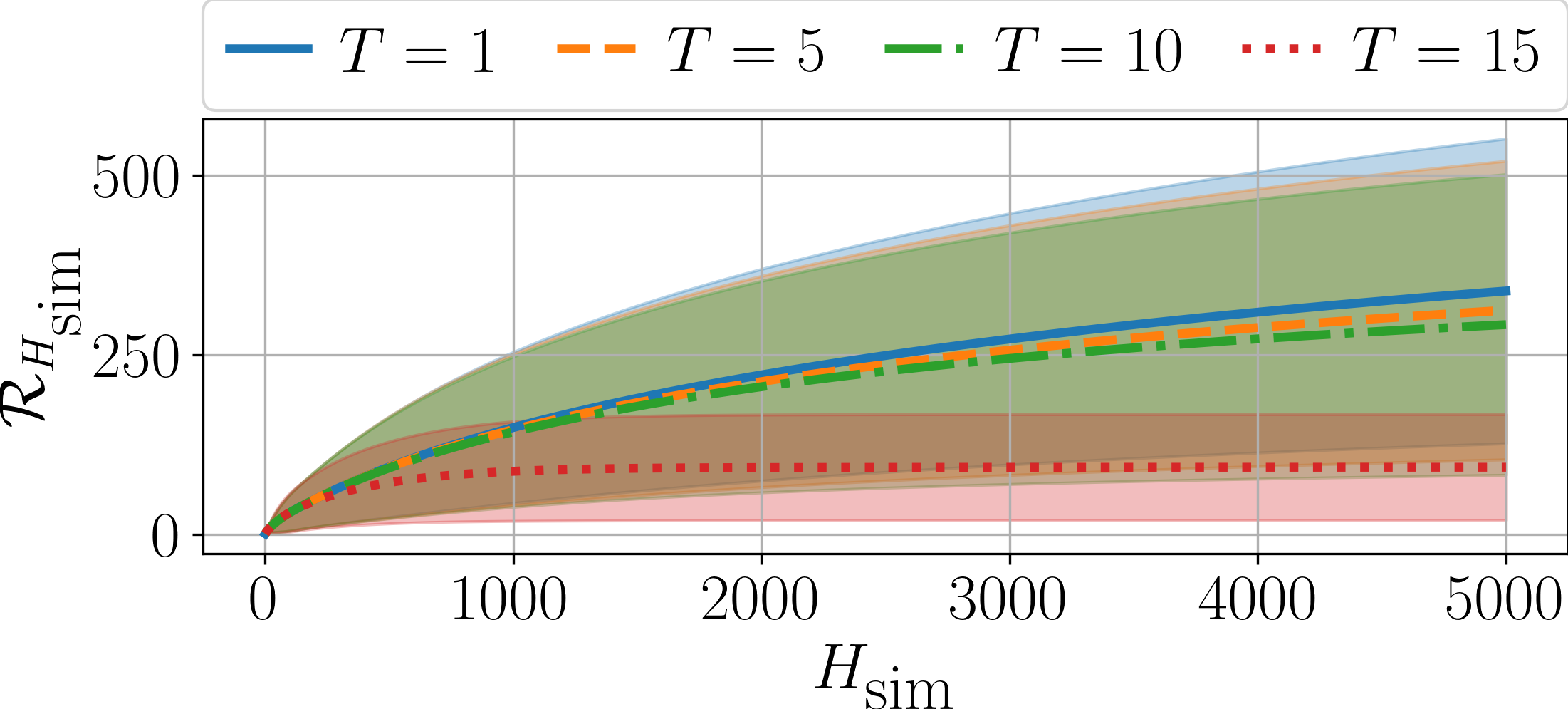}
    \caption{KPI for varying $T$.}
    \label{fig:regret_T}
\end{subfigure}
\begin{subfigure}{.245\linewidth}
    \includegraphics[width=\linewidth]{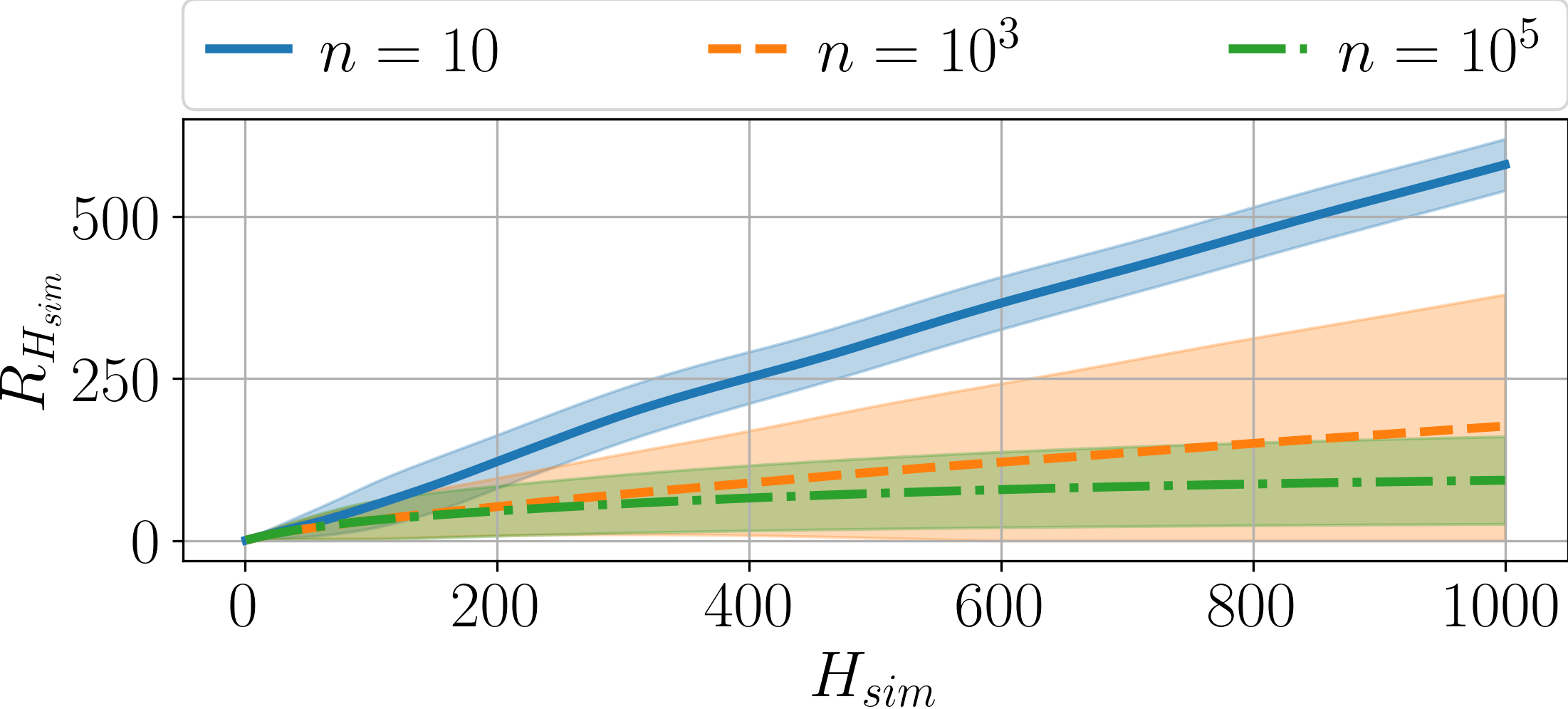}
    \caption{KPI for varying $n$.}
    \label{fig:regret_n}
\end{subfigure}
\caption{(a)--(b) Illustrative examples of  trajectories following \eqref{eq:duffing} when fixing $u=-2$ and $u=2$ respectively (initial conditions marked as blue dots). (c) KPI \eqref{eq:regret} for different values of the predictive horizon, and fixed number of training samples. As expected, a larger value of $T$ improves the performance of the controller, yielding a smaller error. Mean $\pm$ std over 100 random initializations form the unit circle. (c) KPI \eqref{eq:regret} for different values of $n$ and fixed $T$. Mean $\pm$ std over 100 random initializations form the unit circle.}
\end{figure*}
\begin{figure*}
\begin{subfigure}{.24\linewidth}
\centering    \includegraphics[width=\linewidth]{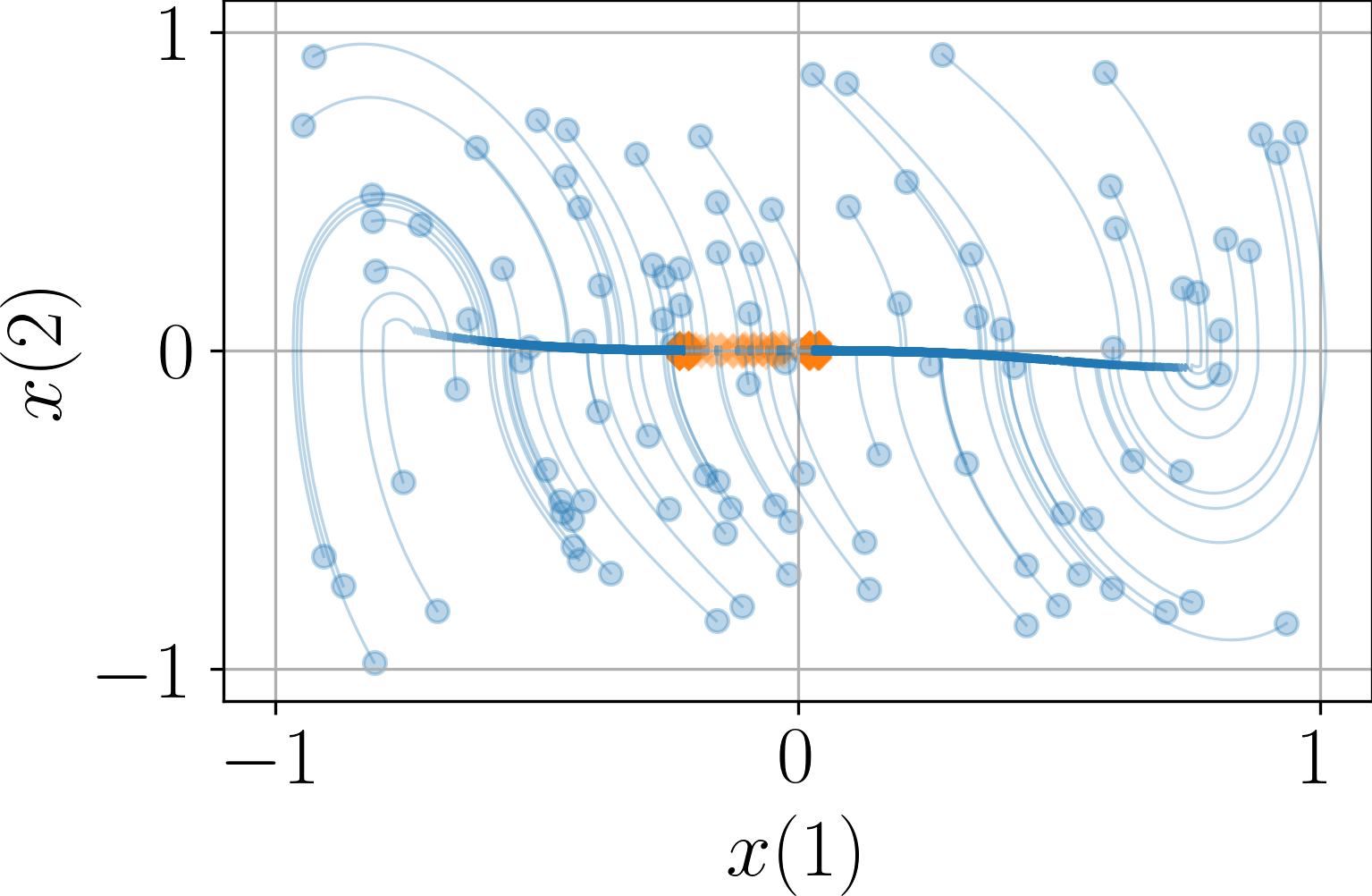}
    \caption{$T=1$.}
\end{subfigure}\hfill
\begin{subfigure}{.24\linewidth}
\centering
\includegraphics[width=\linewidth]{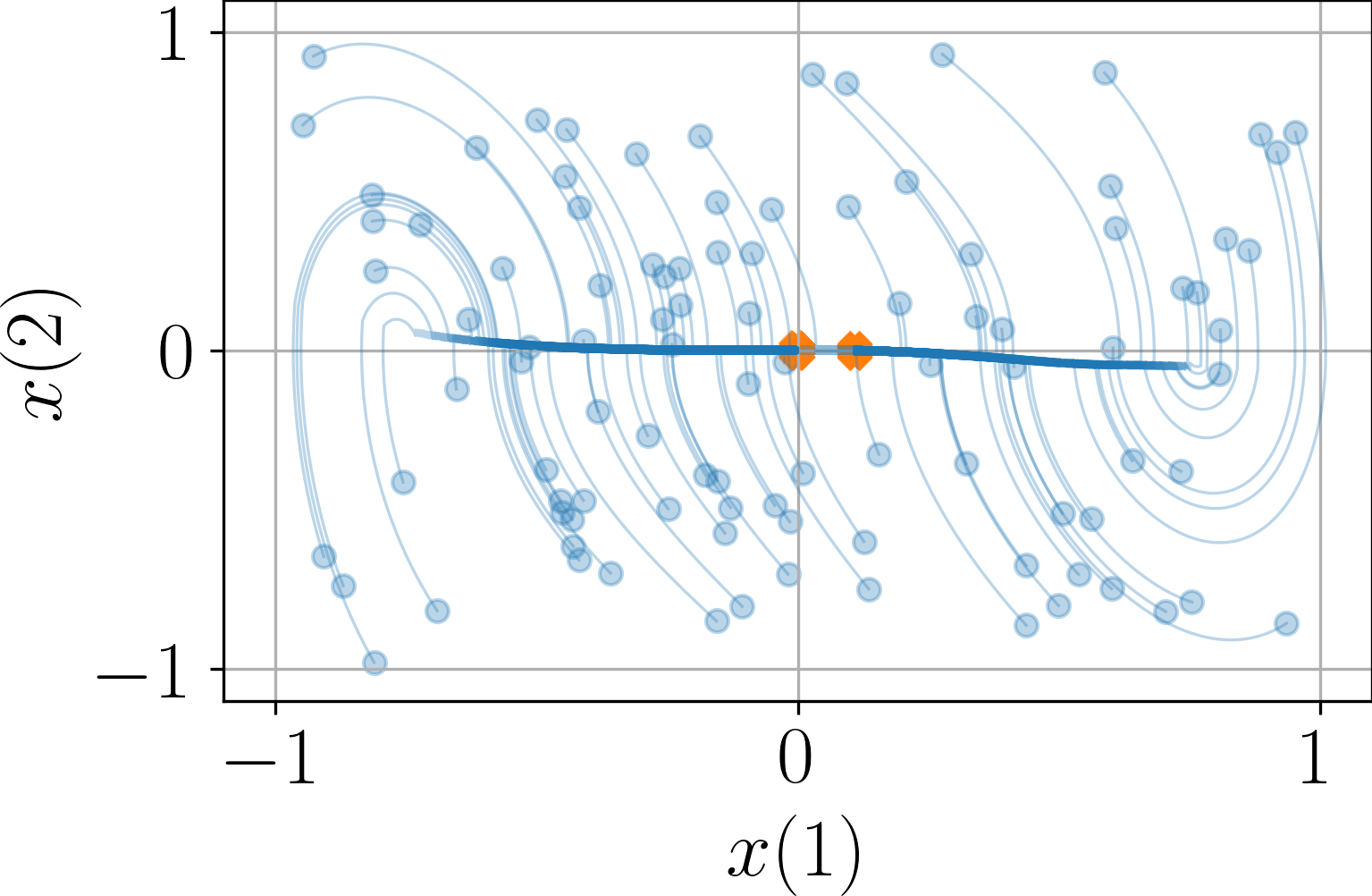}
    \caption{$T=5$.}
\end{subfigure}\hfill
\begin{subfigure}{.24\linewidth}
\centering
\includegraphics[width=\linewidth]{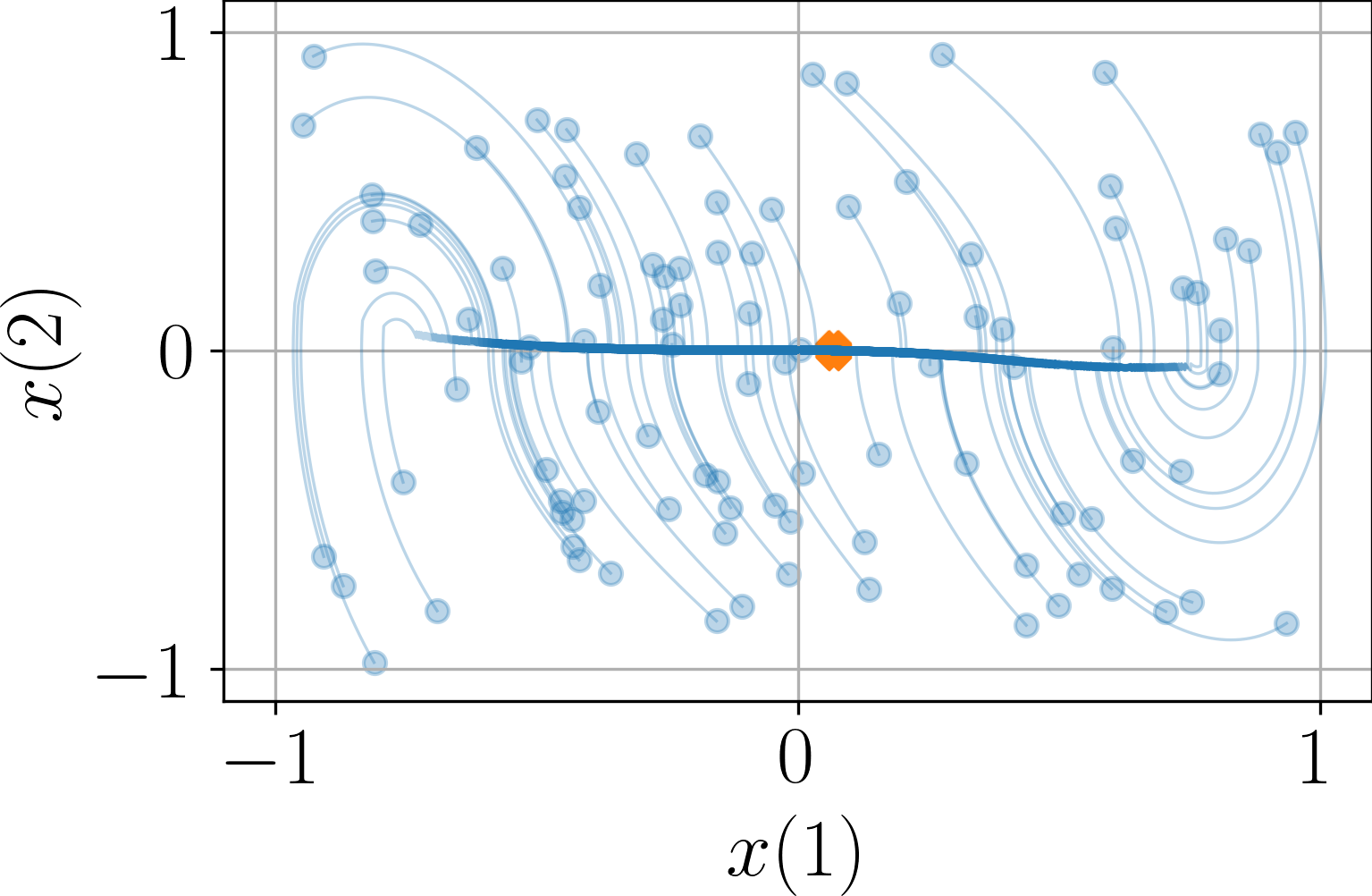}
    \caption{$T=10$.}
\end{subfigure}
\begin{subfigure}{.24\linewidth}
\centering
\includegraphics[width=\linewidth]{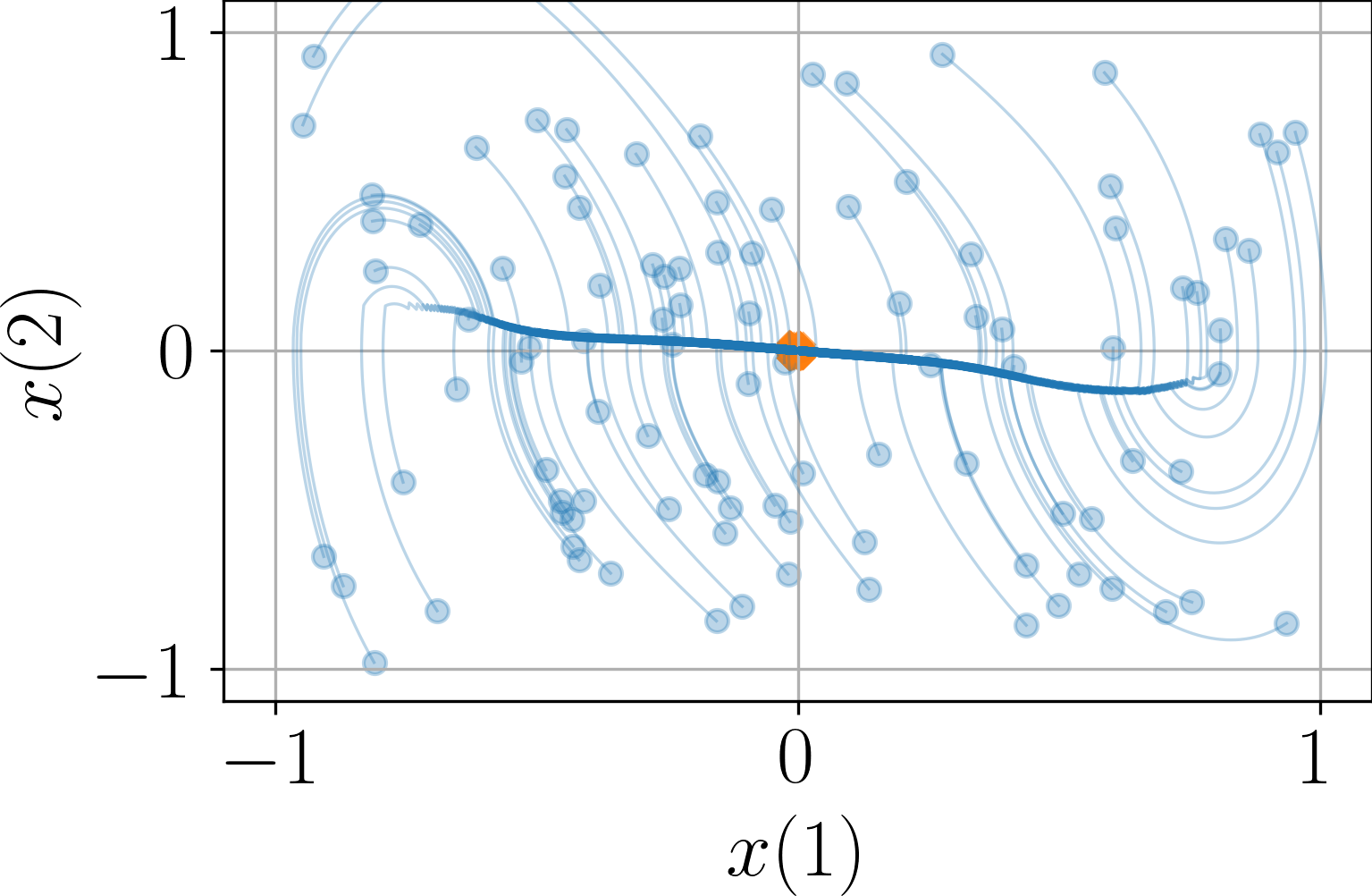}
    \caption{$T=15$.}
\end{subfigure}
\caption{Visualization of 100 trajectories of the dynamics \eqref{eq:duffing}, controlled with our MPC strategy and a symmetric control set, for different values of the MPC predictive horizon $T$ (trajectories simulated for $10^5$ steps for $T<15$ and $5\cdot10^3$ steps for $T=15$). The state at the beginning resp.\ at the end of the simulation is displayed as a blue resp.\ orange dot. As $T$ increases, the performance of the controller improves, (see, e.g., how the system stabilizes at the origin for $T=15$, around the origin for $T=10$, whereas it exhibits multiple attractors for $T=1,\ 5$). The surrogate model is learned with $10^6$ samples.}\label{fig:trajectories_duffing_T}
\end{figure*}
In order to asses the performance of the data-driven MPC, we consider the following key performance indicator (KPI) $\mathcal R_{H_\textrm{sim}}$, measuring the (discounted) distance of the state to the origin, given a simulation horizon $H_\textrm{sim} > 0$:
\begin{equation}
    \mathcal R_{H_\textrm{sim}} = \sum_{t=0}^{H_\textrm{sim}}\lambda^t({\norm{x_t}^2_2} + r\cdot c(u_t)).\label{eq:regret}
\end{equation}
Ideally, we would like this cost to converge, indicating that the system's state, and hence the observable function $\psi$ in the RKHS, stabilize in a sufficiently small neighborhood of the (lifted) origin. This KPI allows to \emph{empirically} validate whether controlling the system in the observable space yields a meaningful behavior in the state space as well (cf.\ Remark \ref{remark:originality_of_ocp}).

\subsection{Symmetric control set}
\label{subsec:symm_ctrl_set}
\begin{figure*}
\begin{subfigure}{.25\linewidth}
\centering    \includegraphics[width=\linewidth]{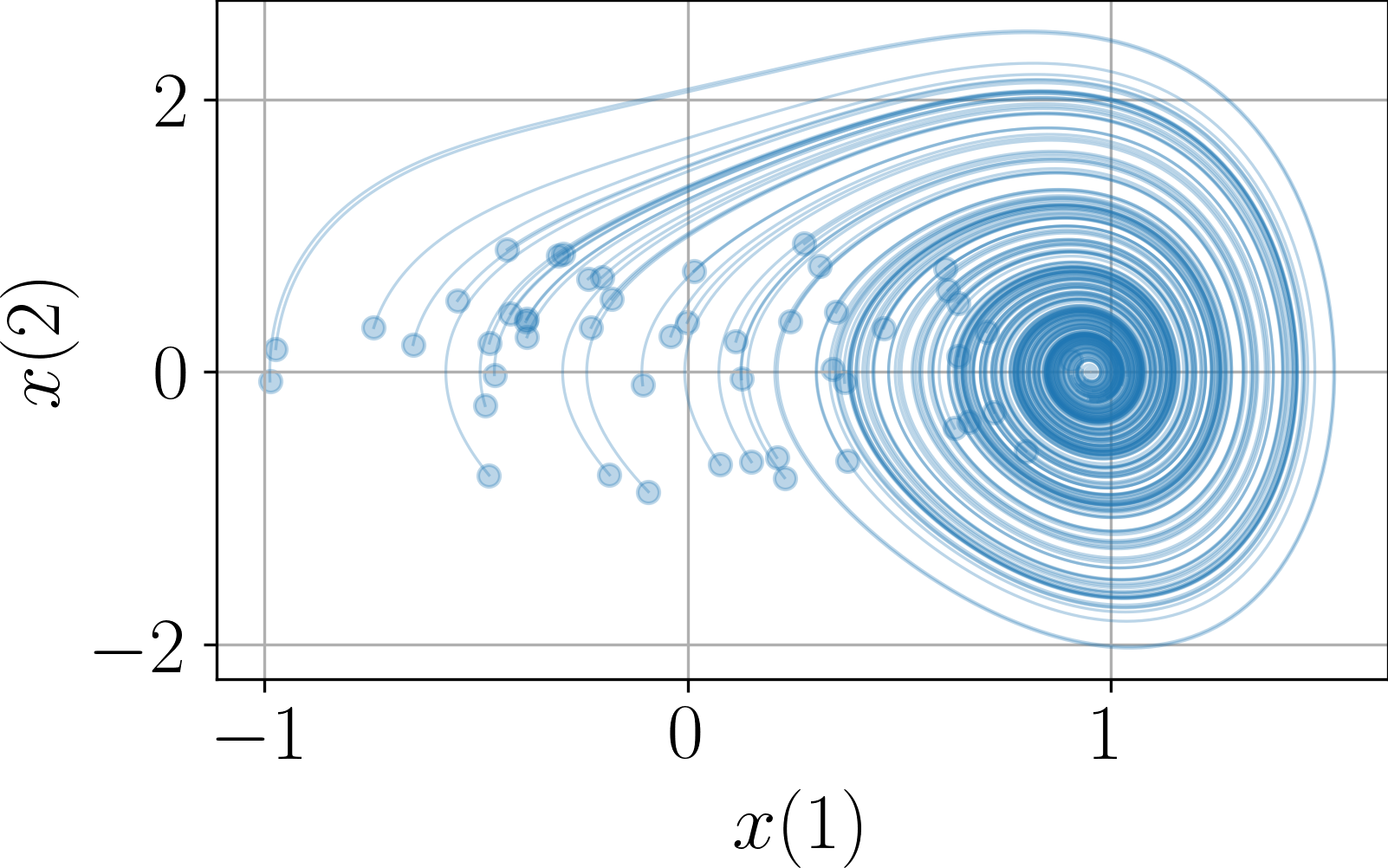}
    \caption{$u=5$.}
\label{fig:sample_trajs_5}
\end{subfigure}\hfill
\begin{subfigure}{.24\linewidth}
\centering
\includegraphics[width=\linewidth]{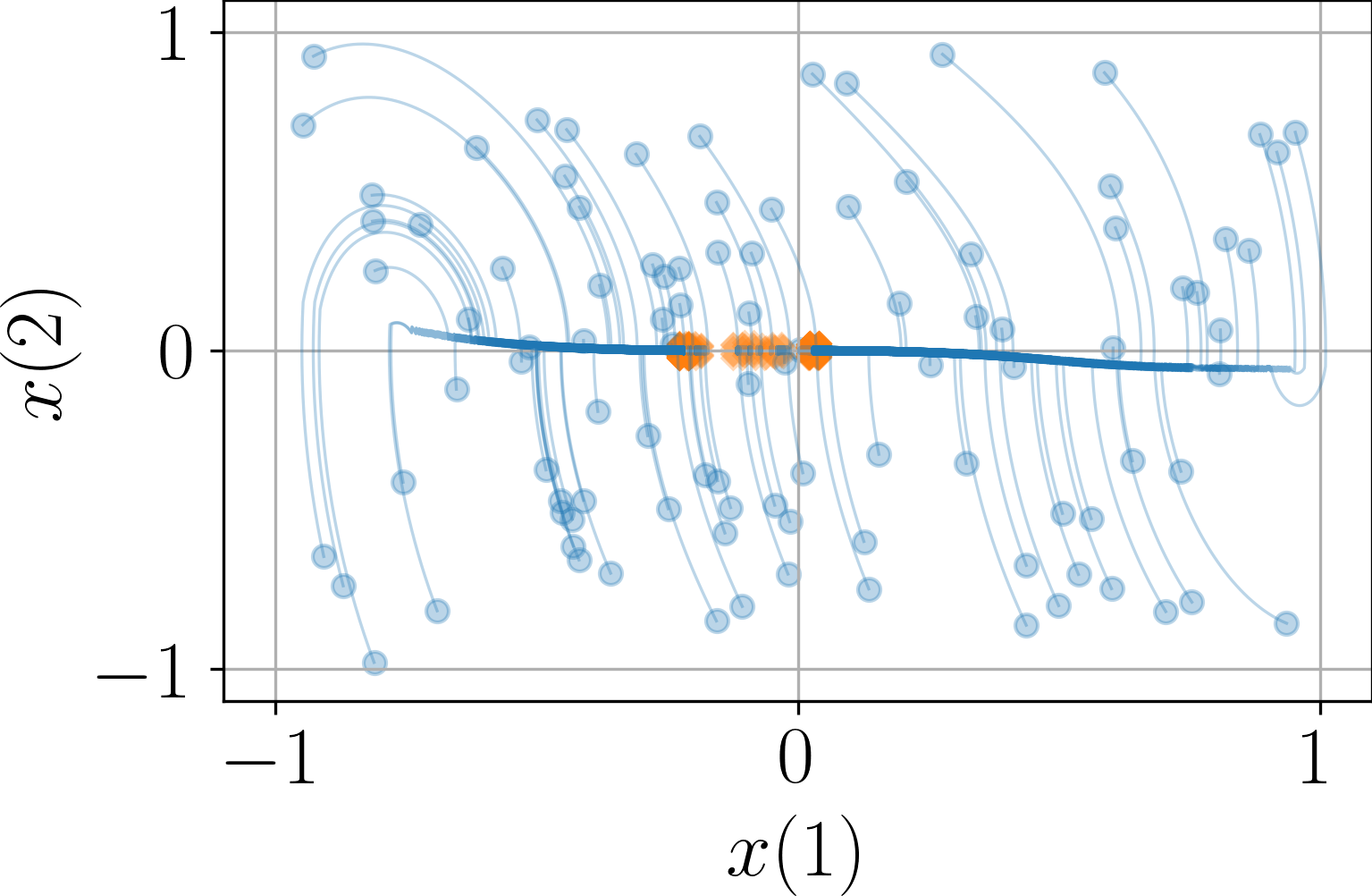}
    \caption{$T=1$.}
    \label{fig:asymm_T_1}
\end{subfigure}\hfill
\begin{subfigure}{.24\linewidth}
\centering
\includegraphics[width=\linewidth]{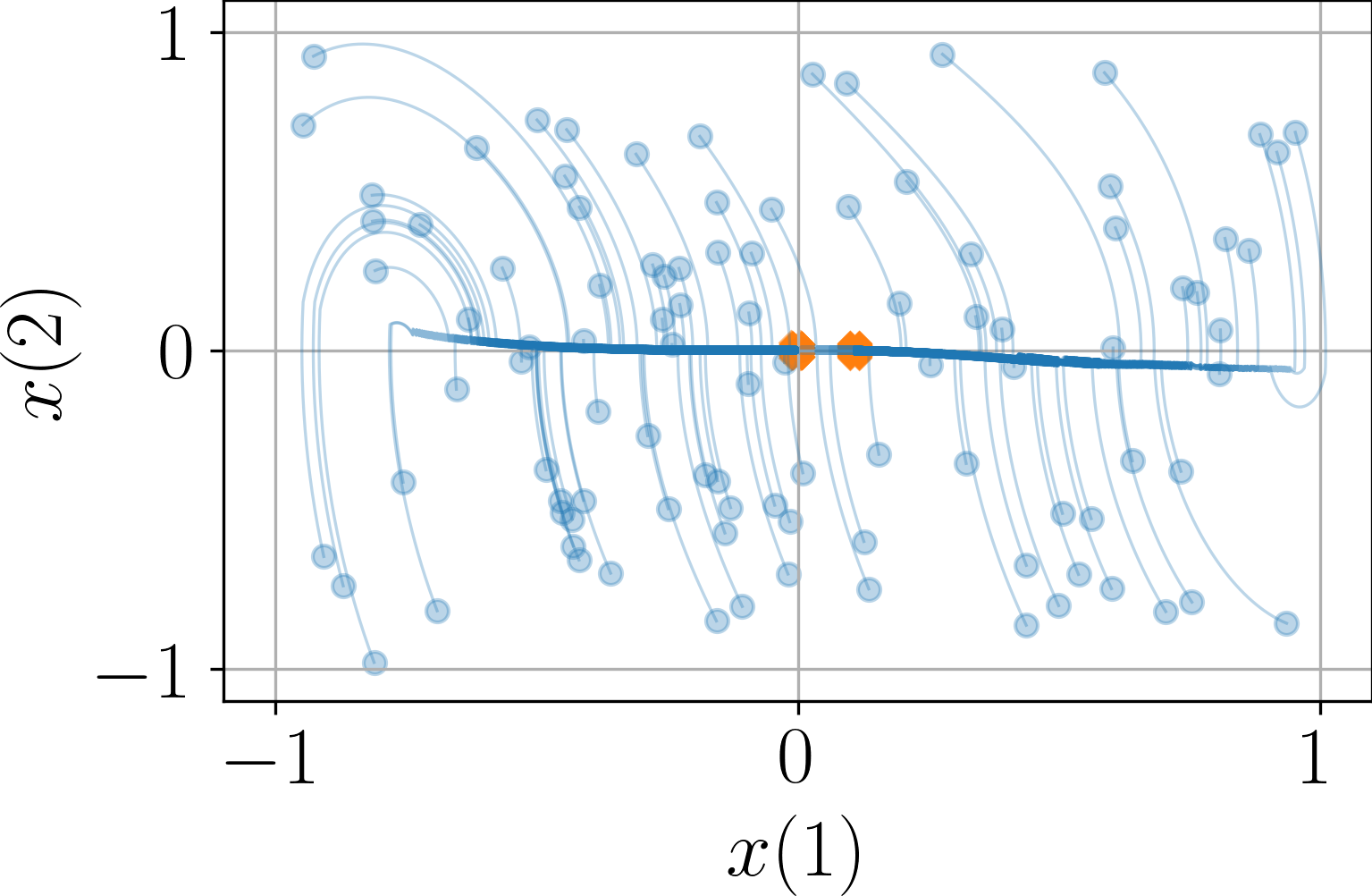}
    \caption{$T=5$.}
        \label{fig:asymm_T_5}
\end{subfigure}
\begin{subfigure}{.25\linewidth}
\centering
\includegraphics[width=\linewidth]{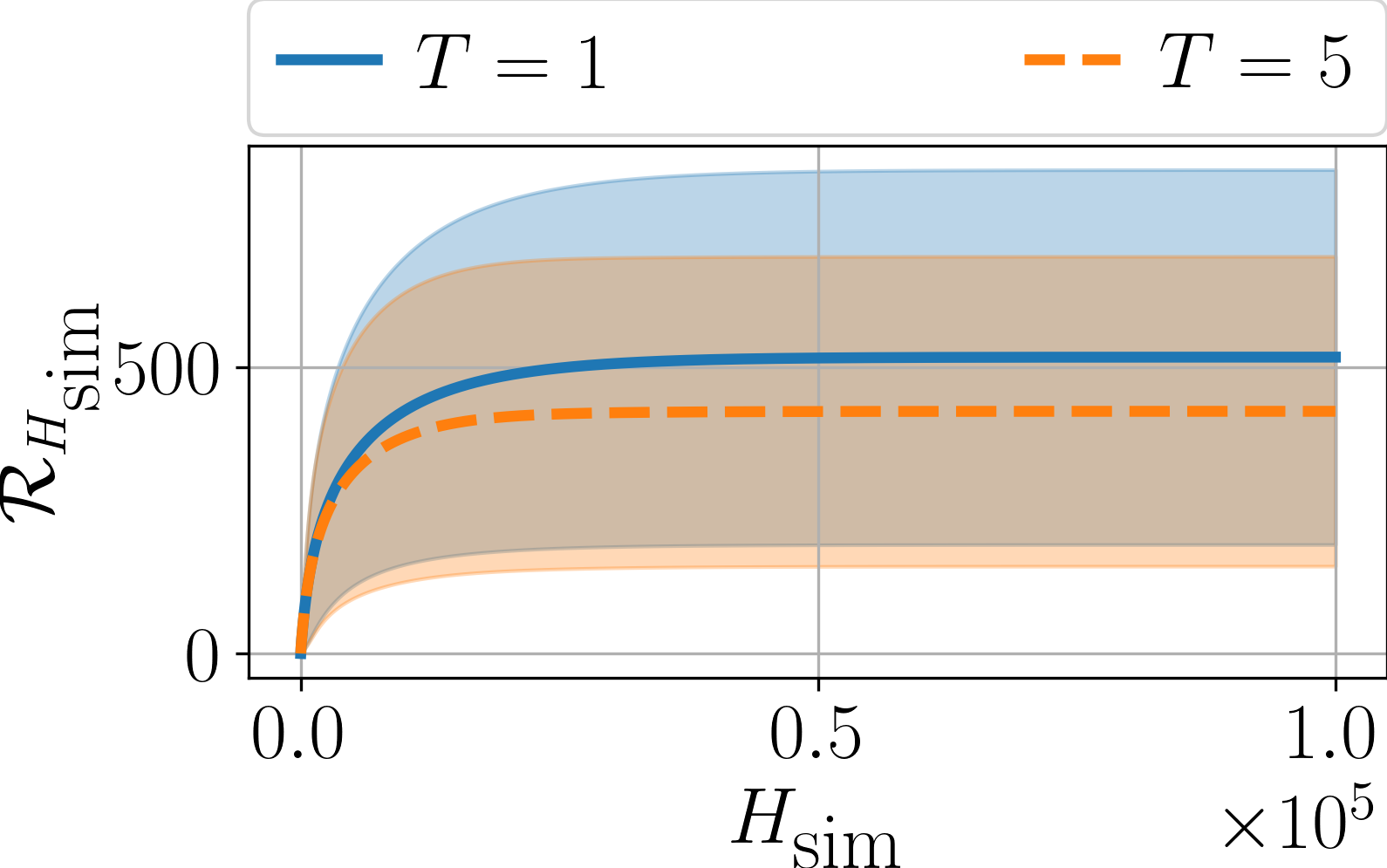}
    \caption{KPI.}
\end{subfigure}
\caption{(a) Illustrative examples of trajectories following \eqref{eq:duffing} when fixing $u=5$ (initial conditions marked as blue dots). (b)--(c) Visualization of 100 closed-loop trajectories of the dynamics \eqref{eq:duffing}, controlled with our MPC strategy and an asymmetiric control set, for different values of the MPC predictive horizon $T$ (trajectories simulated for $10^5$ steps). The state at the beginning resp.\ at the end of the simulation is displayed as a blue resp.\ orange dot. In this scenario, the costly control is used at the beginning of the trajectory to approach the sliding surfaces faster. (d) KPI \eqref{eq:regret} for our closed-loop controller with an asymmetric control set.}\label{fig:trajectories_duffing_T_asymm}
\end{figure*}
We begin by considering a symmetric control set, namely $\mathcal U = \{-2, 2\}$. Note that the number of chosen controls is in line with, e.g., the experiments in \cite{peitz2019koopman}. An illustration of the trajectories of the oscillator, obtained by fixing the control $u$ to these values, is offered in Figs.\ \ref{fig:sample_trajs_-2} and \ref{fig:sample_trajs_2}. In this case, we set $r=0$ in \eqref{eq:stage_cost_duffing} and \eqref{eq:regret}.

Fig.\ \ref{fig:regret_T} shows how the KPI \eqref{eq:regret} changes as a function of the predictive horizon $T$, having fixed the number of training samples $n=10^6$. In this case, the regularization parameter is also fixed, i.e., $\gamma = 10^{-5}$. Conversely, Fig.\ \ref{fig:regret_n} shows the same KPI for increasing $n$ and fixed $T=15$. In this case, the regularization scales as ${1}/{\sqrt{n}}$ (cf.\ Theorem \ref{thm:koopman-sample-rate}). Overall, the behavior of the KPI confirms that increasing the predictive horizon and increasing the number of training samples improve the MPC performance when approaching the origin. 

We offer a qualitative visualization of our control strategy as the predictive horizon $T$ increases, starting from initial conditions sampled uniformly at random from the unit circle. As shown in Fig.\ref{fig:trajectories_duffing_T}, a larger horizon improves the controller's performance, especially in a neighborhood of the origin. We further observe the emergence of sliding surfaces along which the switching closed-loop dynamics move in order to reach the origin. The sliding surfaces on which the closed-loop system moves are determined by the chosen predictive horizon, and in turn affect the algorithm's performance.
\subsection{Asymmetric control set} 
\label{subsec:asymm_ctrl_set}
As a further evaluation, we consider the control set $\mathcal U = \{-2, 2, 5\}$. Fig.\ \ref{fig:sample_trajs_5} shows a qualitative visualization of the oscillator's trajectories obtained by setting $u=5$. We set $r=10^{-6}$ in \eqref{eq:stage_cost_duffing} and \eqref{eq:regret}, and $c(2)=c(-2)=4$, $c(5)=25$. In this case, we fix the number of training samples $n=10^6$ and the regularization $\gamma=10^{-5}$. Comparing the trajectories in Figs.\ \ref{fig:asymm_T_1} and \ref{fig:asymm_T_5} with the ones in Fig.\ \ref{fig:trajectories_duffing_T}, we observe that the additional control allows to move faster towards the switching surfaces in the initial phase of the trajectories.
\section{Conclusions}
\label{sec:conclusion}
In this work, we have studied how Koopman operator learning and model predictive control can be used to learn to control switching nonlinear systems. Starting from the nonlinear dynamics, we have constructed a corresponding Koopman-based, linear switching dynamical system. Such a system is learned from snapshots of the system's evolution, and used to perform model predictive control in closed-loop. From a theoretical point of view, we have quantified the error we incur when learning the Koopman dynamics from finite samples. Moreover, we have studied the performance of the model predictive control algorithm, combined both with exact and inexact Koopman-based predictions. Lastly, we have evaluated the performance of the inexact model predictive control pipeline on a benchmark for nonlinear, data-driven control.

Several questions can be addressed as future research directions. For instance, the theoretical analysis related to Koopman operator regression could be extended to different types of sampling. Moreover, Assumption~\ref{ass:value_function_decay_robust}, here stated in its full generality, could be verified for specific classes of nonlinear systems. 
A further development of this work is related to its application to real controlled systems, which would require to approximate the combinatorial search solving the optimal control problem \eqref{eq:mpc_ocp}, in case quasi-real time performance in closed-loop is required \cite{sager2009reformulations}.
\section*{Acknowledgments}
E.\ C.\ and C.\ M.\ acknowledge support from COST Action InterCoML (CA24136). C.\ M.\ acknowledges the support of the European Commission (grant TraDE-OPT 861137), of the ERC (grant SLING 819789), of the US AFOSR (FA8655-22-1-7034), of the MIUR (PRIN 202244A7YL), and of the project PNRR FAIR PE0000013-SPOKE 10. The research by C.\ M.\ has been supported also by the MIUR Excellence Department Project awarded to DIMA, UniGe, CUP D33C23001110001. C.\ M.\ is member of the GNAMPA (INdAM). O.\ K.\ and L.\ R.\ acknowledge the financial support of the European Commission (Horizon Europe grant ELIAS 101120237). L.\ R.\ also acknowledges the financial support of the Ministry of Education, University and Research (FARE grant ML4IP R205T7J2KP).
\newpage
\appendix
\section{Proofs for Koopman operator extension}
\subsection{Proof of Proposition \ref{lem:linear-extension-H0}}
\label{app:proof-linear-extension}
We first show that for a strictly p.d.\ kernel, for every finite collection $\{(x_i,\alpha_i)\}_{i=1}^m\subset \cX\times\bbR$,
\begin{equation}
\sum_{i=1}^m \alpha_i \psi_{x_i} = 0 \in \cH
\quad \Longrightarrow \quad
\sum_{i=1}^m \alpha_i \psi_{\ff_u(x_i)} = 0 \in \cH.
\label{eq:minimal-welldefinedness}
\end{equation}
Indeed by definition of strictly p.d. kernel we have,
\begin{equation*}
0 = \left \| \sum_{i=1}^m \alpha_i \psi_{x_i} \right\|_\cH^2 = \sum_{i,j=1}^m \left\langle \alpha_i \psi_{x_i}, \alpha_j \psi_{x_j}  \right\rangle_\cH = \sum_{i,j=1}^m \alpha_i \alpha_j k(x_i, x_j) \implies \alpha_1 = \dots = \alpha_m = 0.
\end{equation*}
Therefore,
\begin{equation*}
\sum_{i=1}^m \alpha_i \psi_{\ff_u(x_i)} = 0,
\end{equation*}
as claimed.

Let $h\in\cH_0$. Choose one representation $h=\sum_{i=1}^m \alpha_i \psi_{x_i}$ and define $\widetilde K_u h$ as in~\eqref{eq:linear-extension-definition}. If also
$h=\sum_{j=1}^k \beta_j \psi_{y_j}$, then subtracting gives
$\sum_{i=1}^m \alpha_i \psi_{x_i}-\sum_{j=1}^k \beta_j \psi_{y_j}=0$, hence by~\eqref{eq:minimal-welldefinedness},
\[
\sum_{i=1}^m \alpha_i \psi_{\ff_u(x_i)}-\sum_{j=1}^k \beta_j \psi_{\ff_u(y_j)}=0,
\]
which shows that the definition is independent of the representation. Linearity follows from~\eqref{eq:linear-extension-definition}. Uniqueness holds because any linear map agreeing with $\widetilde K_u$ on the generators $\{\psi_x\}_{x \in \cX}$ must agree with~\eqref{eq:linear-extension-definition} on all finite linear combinations, i.e., \( \widetilde K_u \psi_x = \psi_{\ff_u(x)} \) for every \( x \in \cX \).

\begin{remark}[Strictly p.d.\ kernel]
The minimal assumption necessary for the linear extension to exist and to be unique can be formulated in terms of~\eqref{eq:minimal-welldefinedness} holding for every finite collection $\{(x_i,\alpha_i)\}_{i=1}^m\subset \cX\times\bbR$. However, we use a stronger assumption on the kernel as a clean requirement that might be easier to verify in practice.
\end{remark}
\subsection{Proof of Proposition \ref{prop:linear-extension-2}}
\label{app:proof-linear-extension-2}
    By Proposition \ref{lem:linear-extension-H0}, each $\widetilde K_u:\cH_0\to\cH$ is a linear operator. Moreover, under Assumption \ref{ass:prekoopman-uniformly-bounded}, each $\widetilde K_u:\cH_0\to\cH$ is bounded, with corresponding operator norm upper bounded by $R_u$. Then, according to the bounded linear extension theorem \cite[Section 2.7]{kreyszig1991introductory}, there exists an extension of $\widetilde K_u:\cH_0\to\cH$ to a linear operator $\cK_u: \overline{\cH_0}\to\cH$ which is bounded with the same constant $R_u$. Moreover, since $\overline \cH_0=\cH$, $\cK_u$ is the unique linear extension of $\widetilde K_u$ on $\cH$. 
\section{Proof of Proposition \ref{prop:closed_form_sol}}
\label{app:closed_form_sol}
In what follows, we consider a fixed \( u \in \cU \). To simplify the notation, we drop \( u \) where the meaning is clear from the context, e.g., \( \ff = \ff_u \), \( \cK = \cK_u \), and \( \rho = \rho_u \).
For a fixed set of points \( \left\{ x_1, \dots, x_n \right\} \), we consider the optimization objective \( \widehat \cR_\gamma: \HS(\cH) \to \bbR \) given by~\eqref{eq:emp-reg-risk},
\begin{equation*}
\widehat \cR_\gamma(W) \coloneqq \frac{1}{n} \sum_{i = 1}^n \bigl\| \psi_{\ff(x_i)} - W \psi_{x_i} \bigr\|_\cH^2 + \gamma \norm{W}_{\HS}^2.
\end{equation*}
The objective is Fr{\'e}chet differentiable on the space of Hilbert--Schmidt operator and the gradient w.r.t. \( W \) is given by
\begin{equation}
\nabla \widehat \cR_\gamma(W) = \frac{2}{n} \sum_{i=1}^n \left( W \psi_{x_i} - \psi_{\ff(x_i)} \right) \otimes \psi_{x_i} + 2 \gamma W.
\label{eq:emp-grad}
\end{equation}
Setting the gradient in~\eqref{eq:emp-grad} to \( 0 \), we recover the first-order optimality condition,
\begin{align*}
\frac{1}{n} \sum_{i=1}^n \left( W \psi_{x_i} - \psi_{\ff(x_i)} \right) \otimes \psi_{x_i} + \gamma W = 0.
\end{align*}
Observe that, by linearity,
\begin{equation*}
\frac{1}{n} \sum_{i=1}^n W \psi_{x_i} \otimes \psi_{x_i} = W \left( \frac{1}{n} \sum_{i=1}^n \psi_{x_i} \otimes \psi_{x_i} \right).
\end{equation*}
Therefore, with the following operators
\begin{equation*}
\widehat \Sigma \coloneqq \frac{1}{n} \sum_{i=1}^n \psi_{x_i} \otimes \psi_{x_i} \in \cL(\cH)
\quad \text{and} \quad
\widehat C \coloneqq \frac{1}{n} \sum_{i=1}^n \psi_{\ff(x_i)} \otimes \psi_{x_i} \in \cL(\cH),
\end{equation*}
for the minimizer \( \widehat \cK_\gamma \) we get
\begin{equation*}
\widehat \cK_\gamma \left( \widehat \Sigma + \gamma I \right) = \widehat C.
\end{equation*}
Equivalently, defining $\widehat\Sigma_\gamma:=\widehat \Sigma + \gamma I$,
\begin{equation}
\widehat \cK_\gamma = \widehat C \, \widehat \Sigma_\gamma^{-1}\label{eq:cf_Kempirical}
\end{equation}
as \( \widehat \Sigma_\gamma \) is always invertible. By strong convexity of \( \widehat \cR_\gamma \), \( \widehat \cK_\gamma \) is the unique global minimizer.

To conclude the proof, introduce the following sampling operators:
\begin{align*}
    &\widehat S:\cH \to \mathbb R^n, \widehat Sh = \frac{1}{\sqrt n}[h(x_1), \dots, h(x_n)]^T,\\
    &\widehat Z:\cH \to \mathbb R^n, \widehat Zh = \frac1{\sqrt n} [h(\ff(x_1)), \dots, h(\ff(x_n))]^T.
\end{align*}
The adjoints of $\widehat S$ and $\widehat Z$ are given by
\begin{equation*}
    \widehat S^*a = \frac1{\sqrt n }\sum_{i=1}^n a_i \psi_{x_i},\  \widehat Z^*a = \frac1{\sqrt n }\sum_{i=1}^n a_i \psi_{\ff(x_i)}.
\end{equation*}
Observe that
\begin{equation*}
    \widehat C h = \frac1n\sum_{i=1}^n\left (\psi_{\ff(x_i)}\otimes \psi_{x_i}\right)h= \frac 1n\sum_{i=1}^n\langle h,\psi_{x_i}\rangle_\cH\psi_{\ff(x_i)}= \widehat Z^*\widehat Sh.
\end{equation*}
Moreover, 
\begin{equation*}
    \widehat \Sigma h = \frac 1n\sum_{i=1}^n\left (\psi_{x_i}\otimes \psi_{x_i}\right)h= \frac 1n\sum_{i=1}^n\langle h,\psi_{x_i}\rangle_\cH\psi_{x_i}= \widehat S^*\widehat Sh.
\end{equation*}
Hence, given $x\in\cX$,
\begin{align*}
    \widehat \cK_\gamma \psi_x&= \widehat Z^*\widehat S(\widehat S^*\widehat S+\gamma I)^{-1}\psi_x\\
    &= \widehat Z^*(\widehat S\widehat S^*+\gamma I)^{-1}\widehat S\psi_x\\
    &=\sum_{i=1}^na_i \psi_{\ff(x_i)}
\end{align*}
where $a_i$ is the $i$-th entry of $(\widehat S\widehat S^*+\gamma I)^{-1}\widehat S\psi_x= (H_{nn} + n\gamma I)^{-1}H_{nx}$, the last equality holding by the reproducing property.
\section{Proof of Theorem \ref{thm:koopman-sample-rate}}
\label{app:proof_of_learning_rate}
In what follows, we consider a fixed \( u \in \cU \). To simplify the notation, we drop \( u \) where the meaning is clear from the context, e.g., \( \ff = \ff_u \), \( \cK = \cK_u \), and \( \rho = \rho_u \).

Recall that, according to \eqref{eq:cf_Kempirical},
\begin{equation*}
\widehat \cK_\gamma = \widehat C \, \widehat \Sigma_\gamma^{-1}.
\end{equation*}

Similarly, for the objective
\begin{equation*}
\cR_\gamma(W) \coloneqq \int \bigl\| \psi_{\,\ff(x)} - W \psi_{x} \bigr\|_\cH^2 \, \rho(dx) + \gamma \norm{W}_{\HS}^2, \quad W \in \HS(\cH),
\end{equation*}
the unique global minimizer \( \cK_\gamma \) is given by \( \cK_\gamma \coloneqq C \, \Sigma_\gamma^{-1} \), where
\begin{equation*}
\Sigma \coloneqq \int \psi_{x} \otimes \psi_{x} \, \rho(dx) \in \cL(\cH)
\quad \text{and} \quad
C \coloneqq \int \psi_{\,\ff(x)} \otimes \psi_{x} \, \rho(dx) \in \cL(\cH)
\end{equation*}
and $\Sigma_\gamma:=\Sigma + \gamma I$.
Observe that \( C = \cK \, \Sigma \). Indeed, for any \( h \in \cH \),
\begin{equation*}
Ch = \int \langle h, \psi_x \rangle_\cH \, \psi_{\,\ff(x)} \, \rho(dx) =  \int \langle h, \psi_x \rangle_\cH \, \cK \psi_x \, \rho(dx) = \cK \int \left( \psi_x \otimes \psi_x \right) h \, \rho(dx) = \cK \, \Sigma h.
\label{eq:k-sigma-limit}
\end{equation*}
Hence, \( \cK_\gamma \coloneqq \cK \, \Sigma \Sigma_\gamma^{-1} \), and therefore, under Assumption~\ref{ass:prekoopman-uniformly-bounded},
\begin{equation}
\norm{\cK_\gamma} \leq \norm{\cK} \cdot \underbrace{\norm{\Sigma \Sigma_\gamma^{-1}}}_{\leq 1} \leq R.
\label{eq:ck-gamma-op}
\end{equation}

\paragraph*{Learning bounds} To derive our learning bound, we consider the following decomposition
\begin{equation}
\widehat \cK_\gamma  - \cK  = \underbrace{\widehat \cK_\gamma - \cK_\gamma}_{\mathcal T_\text{sample}} + \underbrace{\cK_\gamma - \cK}_{\mathcal T_\text{approx}},
\end{equation}
and proceed by analyzing each term separately.

\paragraph*{Approximation error}
Define the (scalar) kernel integral operator $L_k:\cL_\rho^2 \to \cL_\rho^2$ by
\[
(L_k g)(x') \coloneqq \int k(x',x)\, g(x)\,\rho(dx),
\quad g \in \cL_\rho^2.
\]
Under Assumption~\ref{ass:bounded-kernel}, $L_k$ is trace-class, self-adjoint, and positive \cite[Proposition~8]{rosasco2010learning}.
In particular, \( L_k \) is compact and admits the spectral decomposition
\begin{equation*}
L_k = \sum_{i\ge 1} \mu_i\, (u_i \otimes u_i)
\quad \text{in } \cL^2_\rho,
\end{equation*}
where $\mu_i > 0$, $\mu_i \searrow 0$, and $\{u_i\}_{i\ge1}$ is an orthonormal system in $\cL^2_\rho$.

Let \( S : \cH \hookrightarrow \cL^2_\rho \) denote the canonical inclusion operator, \( (Sf)(x)=f(x) \). The covariance operator satisfies \( \Sigma = S^* S \). Then the nonzero spectra of \( L_k \) and \( \Sigma \) coincide, and \( \Sigma \) admits the decomposition
\begin{equation*}
\Sigma = \sum_{i\ge1} \mu_i\, (w_i \otimes w_i)
\quad \text{in } \cH,
\end{equation*}
where
\begin{equation*}
w_i \coloneqq \frac{1}{\sqrt{\mu_i}}\, S^* u_i \in \cH.
\end{equation*}
The system \( \{w_i\}_{i\ge1} \) is orthonormal in \( \cH \), satisfies
\begin{equation*}
\Sigma w_i = \mu_i w_i
\quad \text{and} \quad
\overline{\operatorname{Ran}(\Sigma)} = \overline{\operatorname{span}\left\{ w_i : i \ge 1 \right\}}.
\end{equation*}
Extending $\{w_i\}$ to an orthonormal basis of $\cH$ if necessary, any $\cK \in \HS(\cH)$ admits the Hilbert--Schmidt expansion
\begin{equation*}
\cK = \sum_{i,j\ge1} \alpha_{i,j}\, (w_i \otimes w_j),
\qquad
\norm{\cK}_{\HS}^2 = \sum_{i,j} \alpha_{i,j}^2.
\end{equation*}
Using $\Sigma w_j=\mu_j w_j$ and $\Sigma_\gamma^{-1} w_j = (\mu_j+\gamma)^{-1} w_j$, we obtain
\begin{equation*}
\cK_\gamma
= \cK\,\Sigma \Sigma_\gamma^{-1}
= \sum_{i,j\ge 1} \alpha_{i,j}\,\frac{\mu_j}{\mu_j+\gamma}\,(w_i\otimes w_j),
\end{equation*}
so that
\begin{equation*}
\cK_\gamma-\cK
= -\sum_{i,j\ge 1}\alpha_{i,j}\,\frac{\gamma}{\mu_j+\gamma}\,(w_i\otimes w_j).
\end{equation*}
Therefore,
\begin{equation}
\label{eq:approx-error-basic}
\norm{\cK_\gamma-\cK}_{\HS}^2
= \sum_{i,j\ge 1}\alpha_{i,j}^2\left(\frac{\gamma}{\mu_j+\gamma}\right)^2.
\end{equation}
Writing $\beta_{i,j} \coloneqq \alpha_{i,j}\,\mu_j^{1/2-r}$, we have
\begin{equation*}
\cK\,\Sigma^{1/2-r}=\sum_{i,j}\beta_{i,j}\,(w_i\otimes w_j),\
\norm{\cK\,\Sigma^{1/2-r}}_{\HS}^2=\sum_{i,j}\beta_{i,j}^2.
\end{equation*}
Under Assumption~\ref{ass:koopman-source-condition}, the latter is finite and equal to $G^2$.

Write $\alpha_{i,j}=\mu_j^{r-1/2}\beta_{i,j}$. Plugging this into~\eqref{eq:approx-error-basic} yields
\begin{align*}
\norm{\cK_\gamma-\cK}_{\HS}^2
&= \sum_{i,j}\beta_{i,j}^2\,\mu_j^{2r-1}\left(\frac{\gamma}{\mu_j+\gamma}\right)^2 \\
&= \gamma^{2r-1}
\sum_{i,j}\beta_{i,j}^2
\underbrace{
\left(\frac{\gamma}{\mu_j+\gamma}\right)^{3-2r}
\left(\frac{\mu_j}{\mu_j+\gamma}\right)^{2r-1}
}_{\le 1 \text{ for } r \in (1/2, 1]} \label{eq:factor_out_gamma}\\
&\le \gamma^{2r-1}\sum_{i,j}\beta_{i,j}^2
= \gamma^{2r-1}\,\|\cK\,\Sigma^{1/2-r}\|_{\HS}^2 = \gamma^{2r-1} G^2,
\end{align*}
and so \( \norm{\cK_\gamma-\cK}_{\HS} \leq \gamma^{r-1/2}G \).

\paragraph*{Sampling error}
Observe that
\begin{equation*}
\widehat \cK_\gamma - \cK_\gamma
= \widehat C \, \widehat \Sigma_\gamma^{-1} - \cK_\gamma
= ( \widehat C - \cK_\gamma \, \widehat \Sigma - \gamma \cK_\gamma ) \widehat \Sigma_\gamma^{-1}
= \left[ ( \widehat C - \cK_\gamma \, \widehat \Sigma ) - ( C - \cK_\gamma \, \Sigma ) \right] \widehat \Sigma_\gamma^{-1},
\end{equation*}
where we used the identity $K_\gamma(\Sigma + \gamma I) = C$, and so \( C - \cK_\gamma \, \Sigma = \gamma \cK_\gamma\). Therefore,
\begin{equation}
\norm{ \widehat \cK_\gamma - \cK_\gamma }_\HS
\leq \norm{ ( \widehat C - \cK_\gamma \, \widehat \Sigma ) - \left( C - \cK_\gamma \, \Sigma \right) }_\HS \cdot \norm{ \widehat \Sigma_\gamma^{-1} }.
\label{eq:sample-decomp}
\end{equation}
Note that \( \widehat \Sigma \) is a positive semi-definite operator, so
\begin{equation*}
\widehat \Sigma_\gamma \coloneqq \widehat \Sigma + \gamma I \succeq \gamma I
\; \implies \; \norm{ \widehat \Sigma_\gamma^{-1} } \leq \gamma^{-1}.
\end{equation*}
To bound the norm of the first term, consider the random variable \( \xi: \cX \to \HS(\cH) \)
\begin{equation*}
\xi(x) \coloneqq \left( \psi_{\,\ff(x)} - \cK_\gamma \, \psi_x \right) \otimes \psi_x \in \HS(\cH), \quad x \in \cX.
\end{equation*}
Note that
\begin{equation*}
\frac{1}{n} \sum_{i=1}^n \cK_\gamma \psi_{x_i} \otimes \psi_{x_i} = \cK_\gamma \left(\frac{1}{n} \sum_{i=1}^n \psi_{x_i} \otimes \psi_{x_i}\right) = \cK_\gamma \widehat \Sigma,
\end{equation*}
and, similarly, \( \E_\rho \left[ K_\gamma \psi_{x_i} \otimes \psi_{x_i} \right] = \cK_\gamma \Sigma \). Therefore,
\begin{equation*}
( \widehat C - \cK_\gamma \, \widehat \Sigma ) - \left( C - \cK_\gamma \, \Sigma \right) = \frac{1}{n} \sum_{i=1}^n \left( \xi(x_i) - \E_\rho \, \xi \right).
\end{equation*}
Under Assumptions~\ref{ass:prekoopman-uniformly-bounded}, \ref{ass:bounded-kernel} and by inequality~\eqref{eq:ck-gamma-op}, \( \xi \) is uniformly bounded. Indeed,
\begin{equation*}
\norm{\xi(x)}_\HS = \norm{\left( \psi_{\,\ff(x)} - \cK_\gamma \, \psi_x \right) \otimes \psi_x}_\HS =  \norm{\psi_{\,\ff(x)} - \cK_\gamma \, \psi_x}_\cH \cdot \norm{\psi_x}_\cH \leq \kappa^2 \left(1 + R\right) =: M.
\end{equation*}
Therefore, by convexity of expectation and the norm, \( \norm{\E_\rho \, \xi}_\HS \leq M \), so that by the triangular inequality, \( \norm{\xi - \E_{\rho\,} \xi}_\HS \leq 2 M \). Moreover, \( \HS(\cH) \) is a separable Hilbert space. Therefore, by \cite[Theorem 3.5]{pinelis1994optimum}, for any \( \tau > 0 \), 
\begin{equation*}
\mathbb P \left( \left\| \sum_{i=1}^n \left( \xi(x_i) - \E_\rho \, \xi \right) \right\|_{\HS} \ge \tau \right)
\leq 2 \exp \left(-\frac{\tau^2}{2 \cdot n (2 M)^2 }\right).
\end{equation*}
Denoting the right hand side as \( \delta \) and solving for \( \varepsilon \), we get that the following holds with probability at least \( 1 - \delta \),
\begin{equation}
\| (\widehat C - \cK_\gamma \widehat\Sigma) - \left(C-\cK_\gamma\Sigma\right) \|_{\HS}
= \left\| \frac{1}{n} \sum_{i=1}^n \left( \xi(x_i) - \E_\rho \, \xi \right) \right\|_\HS
\leq 2 M \sqrt{\frac{2 \log(2/\delta)}{n}}.
\label{eq:sample-prob}
\end{equation}
Putting~\eqref{eq:sample-prob} into~\eqref{eq:sample-decomp} yields with probability at least \( 1 - \delta \)
\begin{equation*}
\norm{ \widehat \cK_\gamma - \cK_\gamma }_\HS
\leq \frac{2 M}{ \gamma } \sqrt{\frac{2 \log(2 / \delta)}{n}}.
\end{equation*}
\paragraph*{Overall bound and sample rate} We have that, for every \( \delta \in (0, 1) \), with probability at least \( 1 - \delta \),
\begin{align*}
\norm{\widehat \cK_\gamma  - \cK}_\HS
&\leq \norm{\widehat \cK_\gamma  - \cK_\gamma}_\HS + \norm{\cK_\gamma  - \cK}_\HS \\
&\leq \gamma^{r-1/2} G + \frac{2 M}{ \gamma } \sqrt{\frac{2 \log(2 / \delta)}{n}}.
\end{align*}
Choosing \( \gamma = c n^{-1/(2r+1)}\) for some constant \( c > 0 \), we get
\begin{equation}
\norm{\widehat \cK_\gamma  - \cK}_\HS \lesssim \left( G + 2\sqrt{2} M \right) \sqrt{\log(2/\delta)} \cdot n^{-\frac{2r - 1}{4r + 2}}.
\label{eq:learning-rate}
\end{equation}
To conclude, for every \( u \in \cU \) and \( x \in \cX \) it holds
\begin{equation}
\norm{f(\psi_x, u) - \hat f(\psi_x, u)}_\cH = \norm{ (\cK_u - \widehat \cK_u)  \psi_x }_\cH \leq \norm{\cK_u - \widehat \cK_u} \cdot \norm{\psi_x}_\cH \leq \kappa \norm{\cK_u - \widehat \cK_u}_\HS.
\label{eq:pre-uniform}
\end{equation}
Define
\begin{equation*}
G_* \coloneqq \max_{u \in \cU} G_u\,
\quad \text{and} \quad
M_* \coloneqq \max_{u \in \cU} M_u = \kappa^2 \left( 1 + R_* \right).
\end{equation*}
For \( \delta \in (0, 1) \) set \( \delta_u = \delta / \abs{\cU} \), and apply a union bound on~\eqref{eq:learning-rate}. We get
with probability at least \( 1 - \delta \),
\begin{equation*}
\norm{\widehat \cK_\gamma  - \cK}_\HS \leq \left( G_{*} + 2\sqrt{2} M_{*} \right) \sqrt{\log(2 \abs{\cU}/\delta)} \cdot n^{-\frac{2r - 1}{4r + 2}}.
\end{equation*}
Putting this back into~\eqref{eq:pre-uniform} proves the claim.

\section{Value function bounds}
\label{app:value-function-bounds-revisited}

In this appendix we collect auxiliary results used in the proof of Theorem~\ref{thm:suboptimality_gap}.
Throughout, Assumption~\ref{ass:stage_cost_ctrl} holds. Moreover, we let $\mumpck(t, z)$ be the MPC feedback obtained with a horizon $k\in\{2, \dots, T\}$.

\begin{lemma}[Uniform bound]
\label{lem:value-bound}
Let Assumption \ref{ass:stage_cost_ctrl} hold. Define \( C \coloneqq\sum_{j=0}^\infty \lambda_j \). Then, \(\forall t \in \mathbb{N}_0 \),
\( x \in \cX \), and all \( k \ge 2 \), setting $z=\psi_x$, it holds that
\[
V_k(t,z) \le\ C \, \ell\bigl(t, z, \mu_k(t,z)\bigr).
\]
Moreover,
\[
V_2(t,z) \le C \, V_1(t,z).
\]
\end{lemma}

\begin{proof}
By Assumption \ref{ass:stage_cost_ctrl}, for any \( z =\psi_x,\ x\in\cX \) and \( u \in \cU \),
there exists a control sequence \( \bar u \in \cU^\infty \) such that
\[
\ell(j, \bar z_j, \bar u_j) \le \lambda_j \, \ell(t, z, u),
\quad j > t,
\]
where \( \bar z_{j+1} = f(\bar z_j, \bar u_j) \), \( \bar z_t = z \),
and \( \{ \lambda_j \}_{j \ge t} \in \ell^1 \).

According to the definition of $C$,
\begin{equation}
V_k(t,z)
\le V_\infty(t,z)
\le \sum_{j=t}^{\infty} \ell(j, \bar z_j, \bar u_j)
\le C \, \ell(t, z, u).\label{eq:value_function_upperbound}
\end{equation}
Choosing \( u = \mu_k(t,z) \) yields the claim for all \( k \ge 2 \).
For \( k = 2 \), the same argument with \( u = \mu_1(t,z) \)
gives \( V_2(t,z) \leq C \, \ell(t, z, \mu_1(t,z))= C V_1(t,z) \).
\end{proof}

\begin{corollary} For \( (t, z) \in \bbN_0 \times \cH \) and $k\geq 2$,
\[
V_k(t, z) = 0 \implies V_{k'} (t, z) = V_\infty(t, z) = 0,\ \forall k' > k.
\]
\label{cor:v-k-zero}
\end{corollary}
\begin{proof}
By definition of \( V_k \) (see~\ref{eq:def-vT}), there exists \( \tilde u \in \cU^k \) such that \( \cJ_k(t,z, \tilde u) = 0 \) (infinum taken over finite set \( \cU^k \) is always attained). Since \( \ell \ge 0 \), \( \ell(t,z, \tilde u_0) = 0 \). By \eqref{eq:value_function_upperbound}, \(0\leq V_{k'}(t, z) \leq V_\infty(t,z) \leq C \ell(t,z,\tilde u_0) = 0 \), therefore \( V_{k'}(t, z) = V_\infty(t,z) = 0 \).
\end{proof}

\begin{lemma}[Relaxed Lyapunov inequality]
\label{lem:relaxed-lyapunov}
Let \( C \) be as in Lemma~\eqref{lem:value-bound} and assume
\[
T \ge \frac{2 \ln C}{\ln C - \ln(C-1)} .
\]
Define
\[
\alpha \coloneqq 1 - \frac{(C-1)^T}{C^{T-2}} .
\]
Then, for all \( t \in \mathbb{N}_0 \) and \( z \in \cH \),
\[
V_T\bigl(t+1, f(z, \mu_T(t,z))\bigr) - V_T(t,z)
\le - \alpha \, \ell\bigl(t, z, \mu_T(t,z)\bigr).
\]
\end{lemma}
\begin{proof}
Fix \(t \in \bbN_0\) and \(z_t \in \cH\), and define \( u_t \coloneqq \mu_T(t,z_t) \) and \( z_{t+1} \coloneqq f(z_t,u_t) \).  We start by bounding
\(
V_T(t+1,z_{t+1}) - V_{T-1}(t+1,z_{t+1})
\)
in terms of \(\ell(t,z_t,u_t)\).

Fix $k\geq 2$. Firstly note that if, \(\forall (s, z) \in \mathbb N_0 \times \cH,  V_{k}(s, z) = 0 \), by Corollary~\eqref{cor:v-k-zero} \( V_{k+1}(s,z)  = 0 \), hence the inequality
\begin{equation}
V_{k+1}(t+1,z_{t+1}) - V_{k}(t+1,z_{t+1})\leq a \cdot \ell(t, z_t, u_t).\label{eq:decrease_Vk_null}
\end{equation}
holds with any \( a \geq 0\), hence for $a=\frac{(C-1)^{k+1}}{C^{k-1}}$.

Conversely, consider now the case in which $\exists (s', z') \in \mathbb N_0 \times \cH$ s.t.\ $V_k(s', z') > 0$. Note that, according to the contrapositive of Corollary \ref{cor:v-k-zero}, $V_{k}(s, z) > 0 \implies V_{k-1}(s, z) > 0$. Furthermore, since stage costs are nonnegative, \(V_{k-1}(s,z) \le V_k(s,z) \).
Define
\begin{equation}
\label{eq:rk-sup}
r_k \coloneqq \sup \left\{ \frac{V_k(s,z) - V_{k-1}(s,z)}{V_{k-1}(s,z)}: \; (s,z) \in \bbN_0 \times \cH, \, V_k(s, z) > 0 \right\} \in [0,\infty].
\end{equation}

Now pick arbitrary $(s,z)\in\mathbb N_0 \times \mathcal H$. Observe that, by Bellman's optimality,
\begin{equation}
V_k(s,z) = \ell\bigl(s, z, \mu_k(s,z)\bigr) + V_{k-1}\bigl(s+1, f(z, \mu_k(s,z))\bigr).\label{eq:bellman_step}
\end{equation}
By \eqref{eq:bellman_step} and Lemma~\ref{lem:value-bound},
\begin{equation}
V_{k-1}\bigl(s+1, f(z, \mu_k(s,z))\bigr)
\le (C-1) \, \ell\bigl(s, z, \mu_k(s,z)\bigr).\label{eq:value_f_stage_cost_bound}
\end{equation}

\medskip \noindent
By Bellman,
\begin{equation}
V_{k+1}(s,z)
= \min_{u \in \cU} \Bigl\{ \ell(s,z,u) + V_k(s+1,f(z,u)) \Bigr\}
\le \ell\bigl(s,z,\mu_k(s,z)\bigr) + V_k\bigl(s+1,f(z,\mu_k(s,z))\bigr).
\label{eq:bellman-step-2}
\end{equation}
If $V_k\bigl(s+1,f(z,\mu_k(s,z))\bigr) > 0$, using the definition of \(r_k\) at time \(s+1\) and state $f(z, \mu_k(s,z))$,
\begin{equation}
V_k\bigl(s+1,f(z,\mu_k(s,z))\bigr)
\le (1 + r_k) \, V_{k-1}\bigl(s+1,f(z,\mu_k(s,z))\bigr).
\label{eq:single-step-rec}
\end{equation}
Else, if $V_k\bigl(s+1,f(z,\mu_k(s,z))\bigr)=0$, \eqref{eq:single-step-rec} holds trivially. Moreover, from \eqref{eq:bellman_step},
\begin{equation}
V_{k-1}(s+1,f(z,\mu_k(s,z))) = V_k(s,z) - \ell(s,z,\mu_k(s,z)).
\label{eq:bellman_vk_1}
\end{equation}
Starting from \eqref{eq:bellman-step-2}, using \eqref{eq:single-step-rec} and \eqref{eq:bellman_vk_1} yields
\begin{align*}
V_{k+1}(s,z)
&\le \ell\bigl(s,z,\mu_k(s,z)\bigr) + (1 + r_k) \bigl(V_k(s,z) - \ell\bigl(s,z,\mu_k(s,z)\bigr)\bigr) \\
&= (1 + r_k) V_k(s,z)  - r_k \,\ell\bigl(s,z,\mu_k(s,z)\bigr).
\end{align*}
Consequently
\[
V_{k+1}(s,z) - V_k(s,z) \leq r_k \, V_k(s,z)  - r_k \, \ell\bigl(s,z,\mu_k(s,z)\bigr).
\]
Using the bound
\(
V_k(s,z) \le C\,\ell(s,z,\mu_k(s,z))
\) from Lemma~\ref{lem:value-bound}, and multiplying by the coefficient \(- r_k / C \le 0 \) yields
\[
-r_k \,\ell\bigl(s,z,\mu_k(s,z)\bigr)
\le - \frac{r_k}{C} V_k(s,z).
\]
Therefore,
\begin{equation}
V_{k+1}(s,z) - V_k(s, z) \le \left(r_k-\frac{r_k}{C}\right)V_k(s,z)
= \frac{C-1}{C}\,r_k V_k(s,z).
\label{eq:diff-via-r}
\end{equation}
If $V_k(s, z) > 0$, we can divide and obtain
\begin{equation}
    \frac{V_{k+1}(s, z) - V_k(s, z)}{V_k(s, z)}\leq \frac{C-1}{C} \, r_k.
\end{equation}
Note that since $0< V_k(s, z) \leq V_{k+1}(s, z)$, $r_{k+1}$ is well defined. By taking supremum~\eqref{eq:rk-sup} at \( k+1 \), we obtain that 
\begin{equation}
r_{k+1} \leq \frac{C-1}{C}r_k.
\end{equation}
Note that, according to the contrapositive of Corollary \ref{cor:v-k-zero}, $V_k(s, z) > 0 \implies V_{k-1}(s, z) > 0 \implies \dots \implies V_2(s, z) > 0\implies V_1(s, z)>0$. Therefore, $r_2$ is well defined (and so is $r_{k'},\ k'< k$).
Iterating the recursion gives
\begin{equation}
r_k \le \left(\frac{C-1}{C}\right)^{k-2} r_2.
\label{eq:rT-bound}
\end{equation}
The condition \(V_2(s, z)\le C V_1(s, z)\) from Lemma \ref{lem:value-bound} implies 
\begin{equation*}
    \frac{V_2(s, z) - V_1(s, z)}{V_1(s, z)} \leq C -1.
\end{equation*}
which gives \( r_2 \le C - 1 \) by taking the supremum over the set $\{(s,z) \in \bbN_0 \times \cH\ |\ V_{2}(s, z) > 0\}$. Using the upper bound in~\eqref{eq:rT-bound},
\begin{equation*}
r_k \le \left(\frac{C-1}{C}\right)^{k-2}(C-1)= \frac{(C-1)^{k-1}}{C^{k-2}}.
\end{equation*}
Using \eqref{eq:diff-via-r} and \eqref{eq:rT-bound}, we get
\begin{equation}
    V_{k+1}(s,z) - V_k(s, z) \le \frac{C-1}{C}\frac{(C-1)^{k-1}}{C^{k-2}}V_k(s, z)\leq \frac{(C-1)^{k}}{C^{k-1}}V_k(s, z). \label{eq:vk_bound_in_k}
\end{equation}
Conversely, if $V_k(s, z) = 0$, according to the proof of Lemma \ref{lem:value-bound} and Corollary \ref{cor:v-k-zero}, \eqref{eq:vk_bound_in_k} trivially holds.

To conclude, set $k = T-1$, $s = t+1$, $z = z_{t+1}$. Using \eqref{eq:decrease_Vk_null} resp.\  \eqref{eq:vk_bound_in_k} together with \eqref{eq:value_f_stage_cost_bound} yields
\[
V_T(t+1, z_{t+1}) - V_{T-1}(t+1, z_{t+1}) \leq \frac{(C-1)^{T-1}}{C^{(T-1)-1}}V_{T-1}(t+1, z_{t+1})
\le \frac{(C-1)^{T}}{C^{T-2}} \, \ell\bigl(t, z_t, \mu_T(t,z_t)\bigr).
\]
Combining this with the Bellman recursion for \( V_T \) yields
\begin{align*}
    V_T(t, z_t) &= \min_u\{l(t, z_t, u) + V_{T-1}(t+1, f(z_t, u))\}\\
    &= l(t, z_t, \mumpc(t, z_t)) + V_{T-1}(t+1, f(z_t, \mumpc(t, z_t))\\
    &\geq V_T(t+1, f(z_t, \mumpc(t, z_t)))  + \left[1- \frac{(C-1)^{T}}{C^{T-2}}\right]l(t, z_t, \mumpc(t, z_t)),
\end{align*}
which proves the claim.
\end{proof}

\subsection{Proof of Theorem~\ref{thm:suboptimality_gap}}
\label{sec:proof_exact_mpc_suboptimality}
Let \( z_{t+1} = f(z_t, \mu_T(t,z_t)) \).
By Lemma~\ref{lem:relaxed-lyapunov},
\[
\alpha \, \ell\bigl(t, z_t, \mu_T(t,z_t)\bigr)
\le V_T(t,z_t) - V_T(t+1,z_{t+1}).
\]
Summing from \( t=0 \) to \( K-1 \) gives
\[
\alpha \sum_{t=0}^{K-1} \ell\bigl(t, z_t, \mu_T(t,z_t)\bigr)
\le V_T(0,z_0) - V_T(K,z_K)
\le V_T(0,z_0),
\]
since \( V_T(K,z_K) \ge 0 \).
Letting \( K \to \infty \) yields
\[
\alpha \, J_\infty(0, z_0, u_T^{\mathrm{MPC}})
\le V_T(0,z_0)
\le V_\infty(0,z_0),
\]
which concludes the proof.

\section{Robust performance bound}
\subsection{Proof of Proposition \ref{prop:VT_lipschitz_rho}}
\label{app:V_lipschitz_rho}
Fix $t \in \bbN_0$ and $z,z' \in \cD_0$ (see definition in~\eqref{eq:def-cd-0}). For any control sequence $\bar u=(u_0,\dots,u_{T-1}) \in \cU^T$, let $\{z_t\}_{t=0}^{T-1}$ and $\{z'_t\}_{t=0}^{T-1}$ be the trajectories initialized at $z_0=z$ and $z'_0=z'$ and driven by the same sequence $\bar u$ under the same dynamics \( f \).

Note that for every \( u \in \cU \), we have
\begin{equation}
\norm{ f(z, u) - f(z', u) }_\cH = \norm{ \cK_u \left( z - z' \right) }_\cH \leq R_* \norm{ z - z' }_\cH.
\label{eq:ass_lipschitz_f}
\end{equation}

By \eqref{eq:ass_lipschitz_f} and induction,
\begin{equation}
\label{eq:traj_lip}
\norm{z_{t}-z'_{t}}_{\cH}
\le R_*^{\,t} \, \norm{z_0-z_0'}_{\cH}
= R_*^{\,t} \, \norm{z-z'}_{\cH},
\quad t=0,1,\dots,T-1.
\end{equation}
Note that $z_t$ and $z_{t}'$ belong to $\cD_t$. Using \eqref{eq:ass_time_dep_lipschitz_ell} and \eqref{eq:traj_lip},
\begin{align*}
\bigl|\cJ_T(0,z,\bar u)-\cJ_T(0,z',\bar u)\bigr|
&= \left|\sum_{t=0}^{T-1}\bigl(\ell(t,z_{t},u_{t})-\ell(t,z'_{t},u_{t})\bigr)\right| \\
&\le \sum_{t=0}^{T-1} \bigl|\ell(t,z_{t},u_{t})-\ell(t,z'_{t},u_{t})\bigr| \\
&\le \sum_{t=0}^{T-1} L_{t}\,\|z_{t}-z'_{t}\|_{\cH} \\
&\le \left(\sum_{t=0}^{T-1} L_{t}\,R_*^{\,t}\right)\,\|z-z'\|_{\cH}
=: L_{V_T} \|z-z'\|_{\cH}.
\end{align*}
Since the above bound holds for every $\bar u \in \cU^T$, for $u'\in\arg\min_{u} \cJ_T(0, z', u)$, we get
\begin{align}
V_T(0,z)-V_T(0,z')
&= \min_{u}\cJ_T(0,z,u)-\min_{u}\cJ_T(0,z',u) \nonumber\\
&\le \cJ_T(0,z,u')-\cJ_T(0,z',u')\nonumber\\
&\le \lvert \cJ_T(0,z,u')-\cJ_T(0,z',u') \rvert\nonumber\\
&\le L_{V_T} \|z-z'\|_{\cH}.
\label{eq:VT0_lipschitz_with_ct}
\end{align}
Exchanging the roles of $z$ and $z'$, we get that, for every $z,z' \in \cD_0$,
\begin{align}
\lvert V_T(0,z)-V_T(0,z') \rvert
&\le L_{V_T} \|z-z'\|_{\cH}.
\end{align}

Consider now a state $x\in\cX$. Consider the lifted states $f(\psi_x, u)$ and $\hat f(\psi_x, u)$. Recall that 
\begin{align*}
    \norm{f(\psi_x, u)}_\cH =\norm{\cK_u\psi_x}_\cH = \norm{\psi_{\ff_u(x)}}_\cH\leq \kappa\implies f(\psi_x, u)\in\cD_0\\
    \norm{\hat f(\psi_x, u)}_\cH = \norm{\widehat \cK_u\psi_x}_\cH \leq \eta\,  \kappa \implies \hat f(\psi_x, u)\in\cD_0.
\end{align*}
Hence, applying \eqref{eq:VT0_lipschitz_with_ct} at the points $f(\psi_x,u)$ and $\hat f(\psi_x,u)$ gives
\begin{align*}
V_T\bigl(0,f(\psi_x,u)\bigr)-V_T\bigl(0,\hat f(\psi_x,u)\bigr)
&\le \bigl|V_T\bigl(0,f(\psi_x,u)\bigr)-V_T\bigl(0,\hat f(\psi_x,u)\bigr)\bigr| \\
&\le L_{V_T} \|f(\psi_x,u)-\hat f(\psi_x,u)\|_{\cH},
\end{align*}
which is the claimed result with the stated $\rho_T$.
\subsection{Proof of Theorem \ref{thm:alpha-perf-nonhomogeneous-v2}}
\label{appendix:robust_ctrl_results}
We report here the proof of our robust performance bound. Fix a horizon \( T \ge 2 \) and \( n \ge 1 \).  Let the MPC law \( \hat\mu_T \) be computed on the learned model \( \hat f \), according to Algorithm~\ref{alg:mpc_approx}. 

Let \(z_0 =\psi_{x_0}\in\cH_0\) for some $x_0\in\cX$. Define the  true closed-loop trajectory recursively as
\[
z_{k+1} = f\bigl(z_k,\hat\mu_T(k,z_k)\bigr),  \quad k \in \bbN_0.
\]
For brevity, denote $\widehat \alpha \coloneqq \widehat\alpha_T$, $\xi \coloneqq \widehat\xi_T$, and 
\[
V^{(k)} \coloneqq V_T(0,z_k), \quad  \ell_k \coloneqq \ell\bigl(0,z_k,\hat\mu_T(k,z_k)\bigr).
\]

Note that we have \(z_k \in \cH_0\) for all \(k\) according to \eqref{eq:prekoopman}. Hence, using Assumptions~\ref{ass:value_function_decay_robust} and \ref{ass:value_function_lipschitz_cont}  for each \(k\), we get
\begin{align*}
V^{(k+1)} - V^{(k)} &= V_T\bigl(0,f(z_k,\hat\mu_T(k,z_k))\bigr) - V_T(0,z_k) 
\le - \widehat\alpha \,\ell_{k}  + \xi +  \rho_T\!\left( \norm{z_{k+1} - \hat f(z_k,\hat\mu_T(k,z_k))}_\cH \right).
\end{align*}

On the event where \eqref{eq:model-mismatch-bound} holds, the monotonicity of \(\rho\) implies
\[
\rho_T\!\left( \norm{z_{k+1} - \hat f(z_k,\hat\mu_T(k,z_k))}_\cH \right)
= \rho_T\!\left( \norm{f(z_k,\hat\mu_T(k,z_k)) - \hat f(z_k,\hat\mu_T(k,z_k))}_\cH \right)
\le \rho_T\bigl(\varepsilon(\delta)\bigr),
\quad \forall k \in \bbN_0.
\]
Therefore, with probability at least \(1-\delta\),
\begin{equation}
\label{eq:perturbed-one-step-rho}
V^{(k+1)} - V^{(k)}  \le - \widehat\alpha \,\ell_k + \xi+\rho_T\bigl(\varepsilon(\delta)\bigr),
\quad \forall k \in \bbN_0 .
\end{equation}

Now, multiply \eqref{eq:perturbed-one-step-rho} from Assumption \ref{ass:decay_stage_cost} by \(\beta_k > 0\) and sum for \(k = 0,\dots,N-1\). With probability at least \(1-\delta\),
\begin{equation}
\label{eq:discount-telescope-rho}
\sum_{k=0}^{N-1} \beta_k (V^{(k+1)}-V^{(k)})
\le - \widehat\alpha \sum_{k=0}^{N-1} \beta_k \,\ell_k + \xi \sum_{k=0}^{N-1} \beta_k+ \rho\bigl(\varepsilon(\delta)\bigr) \sum_{k=0}^{N-1} \beta_k .
\end{equation}
We can rewrite the left-hand side as follows:
\begin{equation*}
\sum_{k=0}^{N-1} \beta_k (V^{(k+1)}-V^{(k)})
= - \beta_0 V^{(0)} - \sum_{k=1}^{N-1} (\beta_k - \beta_{k-1}) V^{(k)} + \beta_{N-1} V^{(N)}.
\end{equation*}
Since \(\{ \beta_k \}_{k \ge 0}\) is non-increasing, \(\beta_k > 0\), and \(V^{(k)} \ge 0\), it holds
\begin{equation*}
- \sum_{k=1}^{N-1} (\beta_k - \beta_{k-1}) V^{(k)} \ge 0,
\quad \text{and} \quad \beta_{N-1} V^{(N)} \ge 0,
\end{equation*}
and hence
\begin{equation*}
\sum_{k=0}^{N-1} \beta_k (V^{(k+1)}-V^{(k)}) \ge - \beta_0 V^{(0)}.
\end{equation*}
Combining this lower bound with the inequality~\eqref{eq:discount-telescope-rho} gives, 
with probability at least \(1-\delta\),
\begin{equation*}
- \beta_0 V^{(0)} 
\le - \widehat\alpha \sum_{k=0}^{N-1} \beta_k \,\ell_k 
   + \xi \sum_{k=0}^{N-1} \beta_k+ \rho_T\bigl(\varepsilon(\delta)\bigr) \sum_{k=0}^{N-1} \beta_k,
\end{equation*}
and therefore
\begin{equation*}
\widehat\alpha \sum_{k=0}^{N-1} \beta_k \,\ell_k 
\le \beta_0 V^{(0)} + \xi \sum_{k=0}^{N-1} \beta_k + \rho_T\bigl(\varepsilon(\delta)\bigr) \sum_{k=0}^{N-1} \beta_k.
\end{equation*}
We can let \( N \to \infty\) and use Assumption~\ref{ass:decay_stage_cost}, which ensures
\(\sum_{k=0}^{\infty} \beta_k =: B < \infty\), to obtain
\begin{equation}
\label{eq:infinite-sum-rho}
\widehat\alpha \sum_{k=0}^{\infty} \beta_k \,\ell_k
\le \beta_0 V_T(0,z_0) + B \left(\xi + \,\rho_T\bigl(\varepsilon(\delta)\bigr)\right).
\end{equation}

Recall the definition of the infinite-horizon cost associated to the approximate MPC control 
sequence \(\hat \uf^\mpc_T\), from Algorithm~\ref{alg:mpc_approx}:
\begin{equation}
\label{eq:infinite-j-upper-bound}
\cJ_\infty(0,z_0,\hat \uf^\mpc_T) 
\coloneqq \sum_{k=0}^{\infty} \ell\bigl(k, z_k, \hat \mu_T(k, z_k)\bigr)
\le \sum_{k=0}^{\infty} \beta_k \,\ell_k,
\end{equation}
where the last inequality is due to Assumption~\ref{ass:decay_stage_cost}. 
Using \(\beta_0 = 1\) in \eqref{eq:infinite-sum-rho} and the inequality~\eqref{eq:infinite-j-upper-bound}, we obtain
\begin{equation*}
\widehat \alpha \, \cJ_\infty(0,z_0,\hat \uf^\mpc_T) 
\le V_T(0,z_0) + B (\xi + \,\rho_T\bigl(\varepsilon(n,\delta)\bigr)).
\end{equation*}
Dividing by \(\widehat \alpha\) yields, with probability at least \(1-\delta\),
\begin{equation*}
\cJ_\infty(0,z_0,\hat \uf^\mpc_T) 
\le \frac{1}{\widehat \alpha} \, V_T(0,z_0) 
   + \frac{B}{\widehat \alpha} (\xi + \,\rho_T\bigl(\varepsilon(n,\delta)\bigr)).
\end{equation*}
Finally, by Bellman's principle of optimality, \(V_T(0,z_0) \le V_\infty(0,z_0)\), so
\begin{equation*}
\cJ_\infty(0,z_0,\hat \uf^\mpc_T) 
\le \frac{1}{\widehat \alpha} \, V_\infty(0,z_0) 
   + \frac{B}{\widehat \alpha} \,(\xi + \rho_\infty\bigl(\varepsilon(n,\delta)\bigr)),
\end{equation*}
which proves the claim. 

\section{Properties of the stage cost}
\label{appendix:lq_cost_rkhs}
We illustrate here how the stage cost defined in \eqref{eq:stage_cost_duffing} fulfills Assumption~\ref{ass:decay_stage_cost} and \ref{ass:value_function_lipschitz_cont} by invoking Proposition~\ref{prop:VT_lipschitz_rho}.

Fix \( \lambda \in (0, 1) \). Recall that the stage cost as defined as:
\begin{equation*}
\ell(t, z, u) \coloneqq \lambda^t (\hnorm{z - \psi_0}^2 + r\, c(u)), \quad z \in \cH, \; u \in \cU.
\end{equation*}

Note that, for every \( t \in \bbN_0 \) and every \( z \in \cH \),
\begin{equation*}
\ell(t, z, u) = \lambda^t (\hnorm{z - \psi_0}^2 + r\, c(u)) = \lambda^t \ell(0, z, u).
\end{equation*}
Therefore, Assumption \ref{ass:decay_stage_cost} is satisfied with \( \beta_t \coloneqq \lambda^t \) and \( B \coloneqq \sum_{t=0}^\infty \lambda^t = \frac{1}{1 - \lambda} \).

For Assumption~\ref{ass:value_function_lipschitz_cont}, fix \( t \in \{0,\dots ,T-1\} \) and $u \in \cU$. Then \( \ell(0, \cdot, u) \) is Lipschitz on the ball \( \cD_t \) with constant \( L_t = 2 \kappa \left( 1 + \eta R_*^{\,t} \right) \). Indeed, for \( z, z' \in \cD_t \), 
\begin{align*}
\abs{\ell(0, z, u)-\ell(0, z', u)} &= \abs{\hnorm{z - \psi_0}^2-\hnorm{z' - \psi_0}^2 + r\, c(u) - r\, c(u)} \\
&= \abs{\hnorm{z - \psi_0}-\hnorm{z' - \psi_0}} \cdot \left( \hnorm{z - \psi_0} + \hnorm{z' - \psi_0} \right) \\
&\le \hnorm{z-z'} \cdot  2 \kappa \left(\eta R_*^{\,t} +1 \right), 
\end{align*}
where we used the reverse triangle inequality and, for instance, $\hnorm{z - \psi_0}\leq \hnorm{z}+\hnorm{\psi_0}\leq \kappa \eta R_*^t + \kappa$ since \( z \in \cD_t \). Hence, the conditions of Proposition~\ref{prop:VT_lipschitz_rho} are satisfied. Lastly, note that, if $R_*^t<1$, there exists $\rho_\infty$ s.t.\ $\rho_T\leq \rho_\infty$, where
\begin{equation*}
    \rho_{\infty}(r) = \underbrace{\left[\sum_{t=0}^\infty 2\kappa(1+\eta R_*^t)R_*^t\right]}_{<\infty}\cdot \, r.
\end{equation*}

\clearpage
\bibliographystyle{unsrt}
\bibliography{bibliography}

\end{document}